\pdfoutput=1
\documentclass[11pt,reqno,final,oneside,singlespacing]{amsart}

\usepackage{geometry}
\usepackage[utf8]{inputenc}

\usepackage[backend=biber,bibencoding=utf8,style=numeric,
url=false,doi=true,eprint=true,isbn=false,
hyperref=auto,backref=false,maxbibnames=99]{biblatex}
\usepackage{setspace}
\usepackage{indentfirst}

\allowdisplaybreaks[0]
\binoppenalty=\maxdimen
\relpenalty=\maxdimen
\usepackage{mathrsfs,mathtools,amssymb,stmaryrd,tikz-cd}

\theoremstyle{plain}
\newtheorem{introthm}{Theorem}

\newtheorem{theorem}{Theorem}[section]
\newtheorem{lemma}[theorem]{Lemma}
\newtheorem{proposition}[theorem]{Proposition}
\newtheorem{corollary}[theorem]{Corollary}

\theoremstyle{definition}
\newtheorem{definition}[theorem]{Definition}
\newtheorem{example}[theorem]{Example}

\theoremstyle{remark}
\newtheorem{remark}[theorem]{Remark}

\usepackage[notcite,notref,color]{showkeys}

\usepackage[pagebackref=false,hidelinks,unicode=true,bookmarks=true,
linktoc=page,pdfstartview={FitH},final=true]{hyperref}

\let\oldmarginpar\marginpar
\renewcommand\marginpar[1]{\-\oldmarginpar[\raggedleft\footnotesize #1]{\raggedright\footnotesize #1}}

\DeclareMathOperator{\id}{id}
\DeclareMathOperator{\pr}{pr}
\DeclareMathOperator{\sym}{sym}
\DeclareMathOperator{\Hom}{Hom}
\DeclareMathOperator{\End}{End}

\DeclareMathOperator{\coDer}{coDer}

\DeclareMathOperator{\Bott}{Bott}
\DeclareMathOperator{\pbw}{pbw}

\DeclareMathOperator{\CE}{CE}
\DeclareMathOperator{\dR}{dR}
\DeclareMathOperator{\At}{At}

\DeclareMathOperator{\twist}{tw}

\newcommand{\cA}{\mathcal{A}}
\newcommand{\cB}{\mathcal{B}}
\newcommand{\cC}{\mathcal{C}}
\newcommand{\cD}{\mathcal{D}}
\newcommand{\cE}{\mathcal{E}}

\newcommand{\cI}{\mathcal{I}}
\newcommand{\cL}{\mathcal{L}}
\newcommand{\cM}{\mathcal{M}}
\newcommand{\cN}{\mathcal{N}}
\newcommand{\cO}{\mathcal{O}}
\newcommand{\cQ}{\mathcal{Q}}
\newcommand{\cR}{\mathcal{R}}

\newcommand{\cU}{\mathcal{U}}

\newcommand{\sK}{\mathscr{K}}

\newcommand{\frakg}{\mathfrak{g}}
\newcommand{\frakh}{\mathfrak{h}}
\newcommand{\frakM}{\mathfrak{M}}

\newcommand{\NN}{\mathbb{N}}
\newcommand{\ZZ}{\mathbb{Z}}
\newcommand{\RR}{\mathbb{R}}
\newcommand{\CC}{\mathbb{C}}
\newcommand{\KK}{\mathbb{K}}

\newcommand{\argument}{\mathord{\color{black!25}-}}

\newcommand{\degree}[1]{\abs{#1}}
\newcommand{\from}{\leftarrow}

\newcommand{\xto}[1]{\xrightarrow{#1}}
\newcommand{\abs}[1]{\left|#1\right|}
\newcommand{\liederivative}[1]{\cL_{#1}}
\newcommand{\shuffle}[2]{\mathfrak{S}_{#1}^{#2}}
\newcommand{\sections}[1]{\Gamma\big(#1\big)}
\newcommand{\enveloping}[1]{\cU(#1)}
\newcommand{\lie}[2]{[#1,#2]}

\newcommand{\duality}[2]{\left\langle#1\middle|#2\right\rangle}
\newcommand{\XX}{\mathfrak{X}}
\newcommand{\OO}{\Omega}
\newcommand{\inv}{^{-1}}
\newcommand{\dual}{^{\vee}}

\newcommand{\coder}{d^\nabla}
\newcommand{\half}{\frac{1}{2}}
\newcommand{\cohomology}[1]{H^{#1}}
\newcommand{\hypercohomology}[1]{\mathbb{H}^{#1}}

\newcommand{\atiyahcocycle}[2]{\At^{#1}_{#2}}
\newcommand{\atiyahclass}[1]{\alpha_{#1}}

\newcommand{\atiyahcocycleQ}{\atiyahcocycle{\nabla}{(\cM,Q)}}
\newcommand{\atiyahclassQ}{\atiyahclass{(\cM,Q)}}

\newcommand{\zo}{^{0,1}}
\newcommand{\oz}{^{1,0}}

\newcommand{\curvature}{R^{\nabla}}
\newcommand{\smooth}[1]{C^{\infty}({#1})}
\newcommand{\tangent}[1]{T_{#1}}
\newcommand{\pbwcommutator}{C^{\nabla}}
\newcommand{\SdifferentialM}{\liederivative{Q}}
\newcommand{\UdifferentialM}{\liederivative{Q}^\cD}
\newcommand{\boldX}{\boldsymbol X}
\newcommand{\boldY}{\boldsymbol Y}
\newcommand{\Amap}{A^{\nabla}}
\newcommand{\Bmap}{B^{\nabla}}
\newcommand{\Bmapindex}[1]{\Bmap_{#1}}
\newcommand{\atiyahcocyclegQ}{\atiyahcocycle{\nabla}{(\frakg[1],d_{\CE})}}
\newcommand{\pbwF}{\pbw}

\newcommand{\bfemph}[1]{\textit{\textbf{#1}}}
\newcommand{\brevee}{}
\newcommand{\TkM}{T_{\KK}M}
\newcommand{\OmegaF}{\Omega_F}
\newcommand{\deltaa}{\overline{\delta}}
\newcommand{\contraction}[2]{\left\langle#1\middle|#2\right\rangle}
\newcommand{\sign}{\varepsilon}
\newcommand{\prB}{q}

\title[Dg manifolds, formal exponential maps and homotopy Lie algebras]{Dg manifolds, formal exponential maps\\ and homotopy Lie algebras}

\thanks{Research partially supported by NSF grants DMS-1707545 and DMS-2001599.}

\author{Seokbong Seol}
\address{Department of Mathematics, Pennsylvania State University}
\email{sxs803@psu.edu}

\author{Mathieu Stiénon}
\address{Department of Mathematics, Pennsylvania State University}
\email{stienon@psu.edu}

\author{Ping Xu}
\address{Department of Mathematics, Pennsylvania State University}
\email{ping@math.psu.edu}

\begin{document}

\begin{abstract}
This paper is devoted to the study of the relation between
`formal exponential maps,' the Atiyah class, and Kapranov
$L_\infty[1]$ algebras associated with dg manifolds in the $C^\infty$ context.
We prove that, for a dg manifold,
a `formal exponential map'
exists if and only if the Atiyah class vanishes.
Inspired by Kapranov's construction of a homotopy
Lie algebra associated with the holomorphic tangent bundle
of a complex manifold, we prove that the space of vector fields on
a dg manifold admits an $L_\infty[1]$ algebra structure, unique up to
isomorphism, whose unary bracket is
the Lie derivative w.r.t.\ the homological vector field,
whose binary bracket is a 1-cocycle representative of the Atiyah class,
and whose higher multibrackets can be computed by a recursive formula.
For the dg manifold $(T_{X}^{0,1}[1],\bar{\partial})$ arising from
a complex manifold $X$, we prove that this $L_\infty[1]$ algebra structure
is quasi-isomorphic to the standard $L_\infty[1]$ algebra structure
on the Dolbeault complex $\Omega^{0,\bullet}(T^{1,0}_X)$.
\end{abstract}

\maketitle

\tableofcontents

%\linenumbers

\section{Introduction}

This paper, which is a sequel to \cite{MR3319134},
is devoted to the study of some differential geometric
aspects of dg manifolds in the $C^\infty$ context.
Dg manifolds (a.k.a.\ $Q$-manifolds \cite{MR1432574,arXiv:math/0605356,MR1230027})
have increasingly attracted attention recently due to
their relevance in various fields of mathematics, in particular,
mathematical physics. They first appeared in the mathematical physics literature
in the study of BRST operators used to describe gauge symmetries.
They play an essential role in the so called AKSZ formalism in the study of
sigma model quantum field theories \cite{MR1432574,MR2819233}.
They arise naturally in a variety of situations in differential geometry,
Lie theory, representation theory and homotopy algebras
\cite{MR3293862,arXiv:1903.02884,MR2768006,MR2971727,MR3000478}.
They are closely related to the emerging fields of derived differential geometry
\cite{arXiv:2006.01376,arXiv:1212.3745,MR3121621,MR3221297,MR4036665,arXiv:1804.07622,MR2641940}
and higher Lie algebroids \cite{MR4091493,MR3090103,MR2840338,MR3771614,MR3893501,
MR4007376,arXiv:2001.01101,MR1958835,arXiv:1903.02884,MR2768006,MR2223155}
(see also \cite[Letters~7 and~8]{arXiv:1707.00265}).

Recall that a \bfemph{dg manifold} is a $\ZZ$-graded manifold
$\cM$ endowed with a homological vector field,
i.e.\ a degree $+1$ derivation $Q$ of $C^\infty(\cM)$ satisfying $\lie{Q}{Q}=0$.
When the underlying $\ZZ$-graded manifold $\cM$ is a $\ZZ$-graded vector space,
a dg manifold is equivalent to a finite dimensional curved $L_\infty$ algebra
(or more precisely a curved $L_\infty [1]$ algebra).
Any complex manifold naturally gives rise to a dg manifold,
as does any foliation of a smooth manifold.
See Examples~\ref{example-one} and~\ref{example-two}.

The exponential map plays an important role
in classical differential geometry.
In graded geometry, it turns out that a certain `formal exponential map' is more useful.
Let us describe examples,
which illustrate the concept of `formal exponential map' we have in mind.
First of all,
let us recall the relation between exponential map
and \bfemph{Poincaré--Birkhoff--Witt isomorphism} (PBW isomorphism in short)
in classical Lie theory.
Let $G$ be a Lie group and let $\frakg$ be its Lie algebra.
The space $D'_0(\frakg)$ of distributions on $\frakg$
with support $\{0\}$ is canonically identified with the symmetric tensor algebra $S(\frakg)$,
while the space $D'_e(G)$ of distributions on $G$ with support $\{e\}$
is canonically identified with the universal enveloping algebra $\enveloping{\frakg}$.
The classical Lie-theoretic exponential map $\exp:\frakg\to G$,
which is a local diffeomorphism near $0$,
can be used to push forward the distributions on the Lie algebra
to distributions on the Lie group.
The induced isomorphism
$S(\frakg)\cong D'_0(\frakg)\xto{\sim} D'_e(G)\cong\enveloping{\frakg}$
is precisely the symmetrization map realizing the well known PBW isomorphism.
This construction has an analogue for smooth manifolds.
However, it requires a choice of affine connection.
Given a smooth manifold $M$, let $R$ denote its algebra of smooth real-valued
functions $C^\infty(M)$.
Each affine connection $\nabla$ on $M$ determines an exponential map
\begin{equation}\label{eq:exp1}
\exp^{\nabla}: T_M\to M\times M,
\end{equation}
which is a local diffeomorphism of fiber bundles
\[ \begin{tikzcd}
T_M \arrow[ r, "\exp^{\nabla}"] \arrow[d, "\pi"'] 
& M\times M \arrow[d, "\pr_1"] 
\\ M \arrow[r, "\id"'] & M 
\end{tikzcd} \]
from a neighborhood of the zero section of $T_M$
to a neighborhood of the diagonal $\Delta$ in $M\times M$.
The space of fiberwise distributions on the vector bundle $\pi: T_M\to M$
with support the zero section can be identified, as an $R$-coalgebra,
to $\sections{S(T_M)}$.
On the other hand, the space of fiberwise distributions on the fiber bundle
$\pr_1:M\times M\to M$ with support the diagonal $\Delta$ can be identified,
as an $R$-coalgebra, to the space $\cD(M)$ of differential operators on $M$.
Pushing distributions forward through the exponential map \eqref{eq:exp1},
we obtain an isomorphism of $R$-coalgebras
\[ \pbw^{\nabla}:\sections{S(T_M)}\to\cD(M) ,\]
called PBW map in
\cite{MR2989383,MR4271478}.
In other words, $\pbw^{\nabla}$ is the fiberwise $\infty$-order jet
(along the zero section) of the exponential map \eqref{eq:exp1}
arising from the connection $\nabla$.
Therefore, one can consider it as a `formal exponential map'
associated with the affine connection $\nabla$.

We have the following explicit formula for $\pbw^{\nabla}$:
\[ \pbw^{\nabla}(X_0 \odot \cdots \odot X_k)(f) \\
= \left.\frac{d}{dt_0}\right|_0 \left.\frac{d}{dt_1}\right|_0 \cdots
\left.\frac{d}{dt_k}\right|_0
f\big(\exp(t_0 X_0 + t_1 X_1 +\cdots + t_k X_k)\big) ,\]
for all $X_0, X_1,\cdots, X_k \in\Gamma(T_{M})$ and $f\in C^\infty(M)$.

It turns out that the map $\pbw^{\nabla}$ admits a nice
recursive characterization which can be described in a purely algebraic way
\cite{MR2989383,MR4271478} involving the connection $\nabla$,
but not the exponential map \eqref{eq:exp1}.
Therefore, despite the geometric origin of the map $\pbw^{\nabla}$,
this `formal exponential map' still makes sense algebraically in a much wider context.
By way of this purely algebraic description, the `formal exponential map' was
extended to the context of $\ZZ$-\emph{graded} manifolds over
the field $\KK$ (with $\KK=\RR$ or $\CC$) by Liao--Sti\'enon \cite{MR3910470}.
The PBW map:
\begin{equation}\label{defnpbw1}
\pbw^{\nabla}:\sections{S(\tangent{\cM})}\to\cD(\cM)
\end{equation}
arising from an affine connection $\nabla$ on a $\ZZ$-graded manifold $\cM$
can be thought of as the induced formal exponential map
(or the fiberwise $\infty$-order jet) of a `virtual exponential map:'
\begin{equation}\label{eq:expcm1}
\exp^{\nabla}:\tangent{\cM}\to\cM\times\cM.
\end{equation}

Now, let $(\cM,Q)$ be a dg manifold. Then,
both $\Gamma(S(\tangent{\cM}))$ and $\cD(\cM)$ in~\eqref{defnpbw1}
are dg coalgebras over the dg ring $(\smooth{\cM},Q)$
--- see Propositions~\ref{pro:2.2} and~\ref{pro:2.3}.
Here $\big(\sections{S(T_\cM)},\SdifferentialM\big)$
can be understood as the space of fiberwise dg distributions
on the dg vector bundle $\pi:T_\cM\to\cM$
with support the zero section
--- the homological vector field on $T_\cM$ is the complete lift $\hat{Q}$
of the homological vector field $Q\in\XX(\cM)$ \cite{MR3319134,MR4276044}.
On the other hand, $\big(\cD(\cM),\UdifferentialM\big)$
can be identified with the space of fiberwise dg distributions
on the dg fiber bundle $\pr_1:\cM\times\cM\to\cM$
with support the diagonal $\Delta\in\cM\times\cM$
--- the homological vector field on $\cM\times\cM$ is $(Q,Q)$.
Recall that for an ordinary smooth manifold $M$, equipped with a vector field $Q$,
the exponential map \eqref{eq:exp1}
arising from a choice of affine connection $\nabla$ on $M$ identifies
the complete lift\footnote{See \cite{MR0350650}.}
$\hat{Q}\in\XX(T_M)$ of $Q\in\XX(M)$ with the vector field $(Q,Q)\in\XX(M\times M)$
if and only if the connection $\nabla$ is invariant under the flow of $Q$.
In the similar fashion, one may wonder whether the `virtual exponential map' \eqref{eq:expcm1}
is a morphism of dg manifolds.
On the level of fiberwise $\infty$-order jets,
this is equivalent to asking whether the map
$\pbw^{\nabla}:\big(\sections{S(T_\cM)},\SdifferentialM\big)
\to\big(\cD(\cM),\UdifferentialM\big)$
is an isomorphism of dg coalgebras over $(\smooth{\cM},Q)$.
As in classical geometry, one expects that
this would be true if the affine connection $\nabla$ on $\cM$
is invariant under the (virtual) flow of the homological vector field $Q$;
in other words, if the Atiyah class of the dg manifold $(\cM,Q)$ vanishes.

Our first main theorem confirms this assertion:

\begin{introthm}[Theorem~\ref{theorem1}]
Let $(\cM,Q)$ be a dg manifold.
The Atiyah class $\atiyahclassQ$ vanishes
if and only if there exists a torsion-free affine connection $\nabla$ on $\cM$ such that
\[ \pbw^\nabla:\big(\sections{S(T_\cM)},\SdifferentialM\big)
\to\big(\cD(\cM),\UdifferentialM\big) \]
is an isomorphism of dg coalgebras over $(C^\infty(\cM),Q)$.
\end{introthm}

The Atiyah class of a dg manifold was first introduced by Shoikhet
\cite{arXiv:math/9812009} in terms of Lie algebra cohomology
and $1$-jets of tangent bundles, appeared also in
the work of Lyakhovich--Mosman--Sharapov \cite[Footnote~6]{MR2608525},
and was studied systematically in~\cite{MR3319134}.
The Atiyah class of dg manifolds plays
a crucial role in the Kontsevich--Duflo type theorem for dg manifolds
\cite{MR3754617,MR4276044}.
Below we recall its definition in terms of affine connections \cite{MR3319134}.

Let $(\cM,Q)$ be a dg manifold. Given an affine connection $\nabla$ on $\cM$,
consider the $(1,2)$-tensor
$\atiyahcocycleQ\in\sections{\cM;\tangent{\cM}^{\vee}\otimes\End(\tangent{\cM})}$
of degree +1 defined by the relation
\[\atiyahcocycleQ (X,Y) = [Q, \nabla_{X}Y] - \nabla_{[Q,X]}Y -(-1)^{\degree X} \nabla_{X}[Q,Y] ,\]
for any homogeneous vector fields $X,Y\in\XX(\cM)$.
Since $\liederivative{Q}(\atiyahcocycleQ)=0$,
the element $\atiyahcocycleQ$ is a 1-cocycle called the \bfemph{Atiyah cocycle}
associated with the affine connection $\nabla$. The cohomology class
\[ \atiyahclassQ:=[\atiyahcocycleQ]\in\cohomology{1}
\big(\sections{\cM;\tangent{\cM}^\vee\otimes\End(\tangent{\cM})}^\bullet,\cQ\big) \]
does not depend on the choice of connection $\nabla$,
and therefore is an intrinsic characteristic class
called \bfemph{Atiyah class} of the dg manifold $(\cM,Q)$
\cite{MR3319134} --- see Proposition~\ref{Atiyahproperty}.

As shown by the pioneering work of Kapranov \cite{MR1671737,MR2431634,MR2661534},
the Atiyah class of a holomorphic vector bundle gives rise to $L_\infty[1]$ algebras.
These $L_\infty[1]$ algebras play an important role
in derived geometry~\cite{MR2657369,MR2472137,MR2431634}
and the construction of Rozansky--Witten invariants
\cite{MR1671737,MR1671725,MR2661534,arXiv:math/0404360,MR3322372}.

It is natural to expect that the Atiyah cocycle of a dg manifold
gives rise to an $L_\infty[1]$ algebra in a similar fashion.
This is indeed true: the following theorem was announced in~\cite{MR3319134},
but a proof was omitted.
We will give a complete proof in the present paper.

\begin{introthm}[Theorem~\ref{cor:main}]\label{cor:main1}
Let $(\cM,Q)$ be a dg manifold.
Each choice of an affine connection $\nabla$ on $\cM$
determines an $L_\infty[1]$ algebra structure
on the space of vector fields $\XX(\cM)$.
While the unary bracket
$\lambda_1:S^{1}\big(\XX(\cM)\big)\to\XX(\cM)$
is the Lie derivative $\liederivative{Q}$ along the homological vector field,
the higher multibrackets
$\lambda_{k}:S^{k}\big(\XX(\cM)\big)\to\XX(\cM)$,
with $k\geq 2$, arise as the composition
\[ \lambda_{k}: S^{k}\big(\XX(\cM)\big)\to\sections{S^{k}(\tangent{\cM})}\xto{R_{k}} \XX(\cM)\]
induced by a family of sections $\{R_k\}_{k\geq 2}$ of the vector bundles
$S^k(T^\vee_\cM)\otimes T_\cM$ starting with $R_2=-\atiyahcocycleQ$.

Furthermore, the $L_{\infty}[1]$ algebra structures on $\XX(\cM)$
arising from different choices of connections are all canonically isomorphic.
\end{introthm}

The $L_{\infty}[1]$ algebras arising in this way are called
the \bfemph{Kapranov $L_{\infty}[1]$ algebras} of the dg manifold.
Our proof of Theorem~\ref{cor:main1} is very much inspired
by Kapranov's construction \cite[Theorem~2.8.2]{MR1671737}.
Essentially, we endow $\sections{S(\tangent{\cM})}$
with a dg coalgebra structure over $(\smooth{\cM},Q)$
using the PBW map \eqref{defnpbw1} and the dg coalgebra
$\big(\cD(\cM),\UdifferentialM\big)$, whose dual dg algebra
can be considered as a kind of ``the algebra of functions'' on the ``formal neighborhood"
of the diagonal $\Delta$ of the product dg manifold
$\big(\cM\times\cM,(Q,Q)\big)$.
By construction, $\pbw^\nabla$ is a formal exponential map identifying
a `formal neighborhood' of the zero section of $T_\cM$
to a `formal neighborhood' of the diagonal $\Delta$ of the product manifold
$\cM\times\cM$. The dg coalgebra structure on $\cD(\cM)$ associated with
the homological vector field $(Q,Q)$ on $\cM\times\cM$ can be pulled back
through this formal exponential map so as to obtain a dg coalgebra
$(S\big(\XX(M)\big),\delta^{\nabla})$, which in turn induces
an $L_{\infty}[1]$ algebra on $\XX(\cM)$.

The Kapranov $L_{\infty}[1]$ algebra of a dg manifold
as in Theorem~\ref{cor:main1} is completely determined
by the Atiyah $1$-cocycle and the sections
\[ R_{k}\in\sections{S^{k}(\tangent{\cM}^{\vee})\otimes\tangent{\cM}}
\cong\sections{\Hom(S^{k}(\tangent{\cM}),\tangent{\cM})} \]
for $k\geq 3$.
It is thus natural to wonder whether the $R_k$'s can be described explicitly.

For the $L_\infty[1]$ algebra structure on the Dolbeault complex
$(\OO^{0,\bullet}(T\oz_X),\overline{\partial})$ associated with
the Atiyah class of the holomorphic tangent bundle $T_X$ of a Kähler manifold $X$,
Kapranov showed that the multibrackets can be described explicitly
by a very simple formula \cite{MR1671737}: Equation~\eqref{eq:Rk1} below.
Consider the $\CC$-linear extension of the Levi-Civita connection
of the K\"ahler manifold $X$; this is a $T_X^\CC$-connection $\nabla$ on $T_X^\CC$.
Since $X$ is Kähler, $\nabla$ induces a $T_X^\CC$-connection on $T\oz_X$,
also denoted by $\nabla$, which decomposes as the sum
$\nabla=\nabla^{\bar{\partial}}+\nabla\oz$ of the canonical flat $T\zo_X$-connection
$\nabla^{\bar{\partial}}$ on $T\oz_X$ and some
$T\oz_X$-connection $\nabla\oz$ on $T\oz_X$.
Since $\nabla\oz$ is torsion-free and $d^{\nabla\oz} \circ d^{\nabla\oz}=0
\in \Omega^{2, 0}(\End T^{1, 0})$, the curvature
of $\nabla$ is $R^{\nabla}=[d^{\nabla^{\overline{\partial}}},
d^{\nabla\oz}]$,
which
equals to $R_2\in \Omega\zo\big(S^2{(T_X\oz)}\dual\otimes T\oz_X\big)$,
the Dolbeault representative of the Atiyah 1-cocycle of
the holomorphic tangent bundle $T_X$.
Kapranov proved \cite[Theorem~2.6]{MR1671737} that, for $k\geqslant 2$,
the $k$-th multibracket $\lambda_k$ on the Dolbeault complex
$(\OO^{0,\bullet}(T\oz_X),\overline{\partial})$
is the composition of the wedge product
\[ \OO^{0,j_1}(T\oz_X)\otimes\cdots\otimes\OO^{0,j_k}(T\oz_X)
\to\OO^{0,j_1+\cdots+j_k}\big((T\oz_X)^{\otimes k}\big) \]
with the map
\[ \OO^{0,j_1+\cdots+j_k}\big((T\oz_X)^{\otimes k}\big)
\to\OO^{0,j_1+\cdots+j_k+1}(T\oz_X) \]
induced by
\[ R_k\in\Omega\zo\big(S^k{(T\oz_X)}\dual\otimes T\oz_X\big)
\subset\Omega\zo\big(\Hom\big((T\oz_X)^{\otimes k},T\oz_X\big)\big) ,\]
and that, for $k\geqslant 3$,
\begin{equation}\label{eq:Rk1}
R_{k}=d^{\nabla\oz}R_{k-1}
\quad\in\Omega\zo\big(S^{k}{(T\oz_X)}\dual\otimes T\oz_X\big).
\end{equation}
If $X$ is a mere complex manifold rather than a K\"ahler manifold,
the relation between the $R_k$'s is more complicated:
it involves the Atiyah 1-cocycle $R_2$, the curvature of $\nabla\oz$,
and their higher covariant derivatives.
Nevertheless, recursive computations are still possible
as shown in~\cite{MR4271478}.

In the present paper, we prove that a similar characterization
of the higher multibrackets holds for the Kapranov $L_\infty[1]$ algebra
of a \emph{dg} manifold:

\begin{introthm}[Theorem~\ref{theorem2}]
\strut
\begin{enumerate}
\item The sections $R_{n}\in\sections{S^n(T^\vee_\cM)\otimes T_\cM}$,
with $n\geq 3$, are completely determined, by way of a recursive formula,
by the Atiyah cocycle $\atiyahcocycleQ$, the curvature $\curvature$,
and their higher covariant derivatives --- see~\eqref{Rn}.
\item In particular, if $\curvature=0$, then $R_{2}=-\atiyahcocycleQ$
and $R_{n}=\frac{1}{n}\widetilde{\coder}R_{n-1}$, for all $n\geq 3$.
\end{enumerate}
\end{introthm}

Finally, we investigate the Kapranov $L_{\infty}[1]$ algebras
arising from two classes of examples of dg manifolds:
those corresponding to finite dimensional $L_{\infty}[1]$ algebras
as described in Example~\ref{example-one},
and those corresponding to manifolds endowed with integrable distributions,
which include not only foliated manifolds but also complex manifolds
as described in Example~\ref{example-two}.
For the dg manifold $(\frakg[1], d_{\CE})$ associated
with a finite-dimensional $L_\infty[1]$ algebra $\frakg[1]$,
we prove that the multibrackets of the Kapranov $L_{\infty}[1]$ algebra
structure on $\XX(\frakg[1])\cong\Hom\big(S(\frakg[1]),\frakg[1]\big)$
can be expressed in terms of the multibrackets of the $L_{\infty}[1]$ algebra
$\frakg[1]$ --- see Proposition~\ref{thm:Linfty}.
We also compute the Atiyah class of the dg manifold $(\frakg[1],d_{\CE})$
in terms of Chevalley--Eilenberg cohomology of $\frakg[1]$
with values in the tensor product of adjoint and coadjoint modules
$(\frakg[1])^\vee\otimes(\frakg[1])^\vee\otimes\frakg[1]$
--- see Proposition~\ref{pro:Atiyah}.
For the dg manifold $(F[1],d_F)$ arising from
an integrable distribution $F\subseteq\TkM$ on a smooth manifold $M$,
we show that the Kapranov $L_\infty[1]$ algebra structure on $\XX(F[1])$
is quasi-isomorphic to the $L_\infty[1]$ algebra $\OmegaF^\bullet(\TkM/F)$
arising from the Lie pair $(\TkM,F)$, which was studied extensively
in~\cite{MR2989383,MR4271478,MR3877426}. In particular,
for the dg manifold $(T^{0,1}_X[1],\bar{\partial})$
associated with a complex manifold $X$,
the Kapranov $L_\infty[1]$ algebra structure on $\XX(T_X^{0,1}[1])$
is quasi-isomorphic to the $L_\infty[1]$ algebra structure on the
Dolbeault complex $(\OO^{0,\bullet}(T\oz_X),\overline{\partial})$
associated with the Atiyah class of the holomorphic tangent bundle $T_X$.
Moreover, each map $\phi_k$ in the quasi-isomorphism $\{\phi_k\}_{k\geq 1}$ 
is $\OO^{0,\bullet}_X$-multilinear --- see Corollary~\ref{cor:NYC}.

Note that Bandiera \cite{MR3579974,MR3622306} proved that, when $X$ is a K\"ahler manifold,
the Kapranov $L_\infty[1]$ algebra structure on $\OO^{0,\bullet}(T^{1,0}_X)$ 
is homotopy abelian \emph{over the field $\CC$}.
It would be interesting to investigate if 
the $L_\infty[1]$ algebra structure of Theorem~\ref{cor:main1} on the space 
$\XX(\cM)$ of vector fields over a dg manifold $(\cM,Q)$
is homotopy abelian \emph{over the field $\KK$}, possibly 
by extending the techniques developed in~\cite{MR3579974,MR3622306}.

\bigskip
\textbf{Notations and conventions.}
Throughout this paper, the symbol $\KK$ denotes a field either $\RR$ or $\CC$.

We reserve the symbol $M$ to denote a smooth manifold (over $\KK$) exclusively.
The sheaf of smooth $\KK$-valued functions on $M$
is denoted $\cO_{M}=\cO_{M}^{\KK}$.
The algebra of globally defined smooth functions on $M$
is $\smooth{M}=\cO_{M}(M)$.

A $(p,q)$-shuffle is a permutation $\sigma$ of the set $\{1,2,\cdots,p+q\}$
such that $\sigma(1)<\cdots<\sigma(p)$ and $\sigma(p+1)<\cdots<\sigma(p+q)$.
The set of $(p,q)$-shuffles will be denoted by $\shuffle{p}{q}$.

We use Sweedler's (sumless) notation for the comultiplication $\Delta$
in any coalgebra $C$:
\[ \Delta(c)=\sum_{(c)}c_{(1)}\otimes c_{(2)}=c_{(1)}\otimes c_{(2)}
,\quad\forall c\in C. \]

All gradings in this paper are $\ZZ$-gradings and $\cM$ will always be
a finite dimensional graded manifold. Throughout the paper,
`dg' means `differential graded.'

Given a graded vector space $V$, the suspension of $V$ is denoted
by $V[1]$ satisfying $V[1]^{n}=V^{n+1}$. We denote the (internal) degree
of an element $v\in V$ by $\degree{v}$.

Many equations throughout the paper have the following general shape:
\begin{equation}\label{stdform}
A(X_1,X_2,\dots,X_n)=(-1)^{\sum_{(i,j)\in\sK}\degree{X_{\sigma(i)}}
\degree{X_{\sigma(j)}}}B(X_{\sigma(1)},X_{\sigma(2)},\dots,X_{\sigma(n)})
,\end{equation}
where $X_1,X_2,\dots,X_n$ is a finite collection of $\ZZ$-graded objects;
$\sigma$ is a permutation of the set of indices $\{1,2,\dots,n\}$;
$\sK$ is the set of couples $(i,j)$ of elements of $\{1,2\dots,n\}$
such that $i<j$ and $\sigma(i)>\sigma(j)$;
and $A$ and $B$ are $n$-ary operations on the $\ZZ$-graded objects
$X_1,X_2,\dots,X_n$ whose output is an object of degree
$\degree{X_1}+\degree{X_2}+\cdots+\degree{X_n}$.
The factor
$(-1)^{\sum_{(i,j)\in\sK}\degree{X_{\sigma(i)}}\degree{X_{\sigma(j)}}}$
appearing in the right hand side of~\eqref{stdform}
is called the \emph{Koszul sign} of the permutation $\sigma$ of the graded objects
$X_1,X_2,\dots,X_n$. It will customarily be abbreviated as $\sign$
since its actual value --- either $+1$ or $-1$ --- can be recovered
from a careful inspection of both sides of the equation.
We will also use the more explicit abbreviation $\sign(X_1,X_2,\cdots,X_n)$
if the collection of $\ZZ$-graded objects begin permuted
is not immediately clear.
As explained by Boardman in~\cite{MR209338},
this sign is mostly inconsequential and it is not necessary
to devote much attention or thought to it.
In fact, the right hand side of \eqref{stdform} can be a sum of several terms
so it would be more correct to say that the general shape of the equations is
\[ A(X_1,X_2,\dots,X_n)=\sum_k (-1)^{\sum_{(i,j)\in\sK_k}
\degree{X_{\sigma_k(i)}}\degree{X_{\sigma_k(j)}}}
B_k(X_{\sigma_k(1)},X_{\sigma_k(2)},\dots,X_{\sigma_k(n)}). \]

\section{Preliminaries}

\subsection{dg manifolds}
Let $M$ be a smooth manifold over $\KK$, and $\cO_{M}$ be the sheaf
of $\KK$-valued smooth functions over $M$. A \bfemph{graded manifold} $\cM$
with support $M$ consists of a sheaf $\cA$ of graded commutative
$\cO_{M}$-algebra on $M$ such that there is a $\ZZ$-graded vector space $V$ satisfying
\[\cA(U)\cong \cO_{M}(U)\otimes_{\KK} \Hom_{\KK}(S(V),\KK)
\cong \cO_{M}(U)\otimes_{\KK} \widehat{S}(V^{\vee})\]
for sufficiently small open set $U \subset M$.
The global section of the sheaf $\cA$ will be denoted by $\smooth{\cM}=\cA(M)$.
We say a graded manifold $\cM$ is finite dimensional if $\dim M<\infty$ and $\dim V <\infty$. 
Throughout this paper, graded manifold $\cM$ will always be finite dimensional.

\begin{remark}
In the literature, the sheaf of functions $\cA$ over a graded manifold 
is defined by $\cA(U)=\cO_M(U)\otimes_{\KK}S(V^\vee)$ for sufficiently 
small open subsets $U$ of $M$.
Here, however, we allow for formal power series rather than polynomials:
$\cA(U)=\cO_M(U)\otimes_{\KK}\widehat{S}(V^\vee)$ for sufficiently 
small open subsets $U$ of $M$.
Consequently, when we write `dg manifold' $(\cM,Q)$, we actually mean 
formal dg manifold in Kontsevich's sense \cite[Section~4.1]{MR2062626}.
\end{remark}

By $\cI_{\cA}$, we denote
the sheaf of ideal of $\cA$ consisting of functions vanishing
at the support $M$ of $\cM$. That is,
for sufficiently small $U\subset M$,
\[ \cI_{\cA} (U)\cong \cO_{M}(U)\otimes_{\KK}\widehat{S}^{\geq 1}(V^{\vee}). \]

Given graded manifolds $\cM=(M,\cA)$ and $\cN=(N,\cB)$,
a \bfemph{morphism} $\cM\to\cN$ of graded manifolds consists of
a pair $(f,\psi)$, where $f:M\to N$ is a morphism of smooth manifolds
and $\psi:f^{*}\cB\to\cA$ is
a morphism of sheaves of graded commutative $\cO_{M}$-algebras such that
$\psi(f^{*}\cI_{\cB})\subset\cI_{\cA}$. We often
use the notation $\phi: \cM \to \cN$ to denote such a
morphism. Then $\psi=\phi^{*}$. Also, we write
$\phi^{*}:\smooth{\cN}\to\smooth{\cM}$ to denote the morphism
on global sections.
Note that the condition $\psi(f^{*}\cI_{\cB})\subset\cI_{\cA}$
is equivalent to $\psi$ being continuous w.r.t\ the $\cI$-adic topology.

Vector bundles in the category of graded manifolds are called
\bfemph{graded vector bundles}.
Given a graded vector bundle $\Phi:\cE\to\cM$, a \bfemph{section}
$s:\cM\to\cE$ of $\cE$ over $\cM$ is a morphism of graded manifolds
such that $\Phi\circ s=\id_{\cM}$. We write the $\smooth{\cM}$-module
of sections of $\cE$ over $\cM$ by the usual notation
$\Gamma(\cE)=\Gamma(\cM;\cE)$.

For a graded manifold $\cM$ with support $M$, its tangent bundle
$\tangent{\cM}$ is a graded manifold with support $T_M$
and is a graded vector bundle over $\cM$.
Its sections are called \bfemph{vector fields} on $\cM$ and
the space of vector fields $\sections{\cM;\tangent{\cM}}=\sections{\tangent{\cM}}$
can be identified with that of graded derivations of
$C^\infty (\cM)$.
We also write $\sections{\cM;\tangent{\cM}}=\XX(\cM)$.
Observe that $\XX(\cM)$ admits a Lie algebra 
structure, whose Lie bracket coincides with the graded commutator
\[[X,Y]=X\circ Y -(-1)^{\degree{X}\cdot \degree{Y}}Y\circ X\]
for homogeneous elements $X,Y\in\XX(\cM)$
regarded as derivations of $\smooth{\cM}$.
Indeed $\tangent{\cM}$ is a graded Lie algebroid \cite{MR2534186}.

A \bfemph{differential graded manifold} 
(dg manifold in short) is a graded manifold 
$\cM$ together with a homological vector field,
i.e.\ a vector field $Q\in\XX(\cM)$ of degree 
$+1$ satisfying $[Q,Q]= Q\circ Q+ Q\circ Q=0$.
For a dg manifold $(\cM, Q)$, its tangent bundle $\tangent{\cM}$ is
naturally a dg manifold, with the homological vector field being
the complete lift\footnote{It is also called tangent lift in the literature 
\cite{MR3319134,MR3754617}.} of $Q$, and in fact $\tangent{\cM}$ is a dg Lie algebroid
over $\cM$ \cite{MR2534186,MR3319134}.

\begin{example}\label{example-one}
Let $\frakg$ be a finite dimensional Lie algebra.
Then $(\frakg[1], d_{\CE})$ is a dg manifold --- 
its algebra of functions is $C^\infty(\frakg[1])\cong \Lambda^\bullet\frakg^\vee$
and its homological vector field $Q$ is the Chevalley--Eilenberg differential $d_{\CE}$.

This construction admits an `up to homotopy' version:
Given a $\ZZ$-graded finite dimensional vector space $\frakg=\bigoplus_{i\in\ZZ}\frakg_i$,
the graded manifold $\frakg[1]$ is a dg manifold, i.e.\ admits a homological vector field,
if and only if $\frakg$ admits a structure of curved $L_\infty$ algebra.
\end{example}

\begin{example}\label{example-two}
Let $M$ be a smooth manifold. Then $(T_M[1],d_{\dR})$ is a dg manifold ---
its algebra of functions is $C^\infty(T_M[1])\cong\Omega^\bullet(M)$
and its homological vector field $Q$ is the de Rham differential $d_{\dR}$.
Likewise, a complex manifold $X$ gives rise to a dg manifold 
$(T^{0,1}_X[1],\bar{\partial})$
whose algebra of functions $C^\infty(T^{0,1}_X[1])$ is $\Omega^{0,\bullet}(X)$
and whose homological vector field $Q$ is the Dolbeault operator $\bar{\partial}$.
\end{example}

\begin{example}\label{example-three}
Let $s$ be a smooth section of a vector bundle $E\to M$.
Then $(E[-1],\iota_s)$ is a dg manifold --- 
its algebra of functions is
$C^\infty(E[-1])\cong\sections{\Lambda^{-\bullet}E\dual}$
and its homological vector field is $Q=\iota_s$, the interior product with $s$.
This dg manifold can be thought of
as a smooth model for the (possibly singular) intersection of $s$
with the zero section of the vector bundle $E$,
and is often called a `derived intersection', or
a \emph{quasi-smooth derived manifold} \cite{arXiv:2006.01376}.
\end{example}

Both situations in Example~\ref{example-two} are special instances
of Lie algebroids, while Example~\ref{example-three} is a special
case of derived manifolds \cite{arXiv:2006.01376}.

\subsection{Atiyah class}
Let $\cM$ be a graded manifold and $\cE$ be
a graded vector bundle over $\cM$. We say a $\KK$-linear map
\[\nabla:\XX (\cM)\otimes_{\KK} \sections{\cE} \to \sections{\cE}\]
of degree $0$ is a \bfemph{linear connection} on $\cE$ over $\cM$ if the following
axioms are satisfied:
\begin{enumerate}
\item $\smooth{\cM}$-linear in the first argument: $\nabla_{fX}s=f \nabla_X s$.
\item $\nabla_X$ is a derivative in the second argument: $\nabla_X (f s)
= X(f) s+(-1)^{\degree f\cdot\degree X}f \nabla_X s$,
\end{enumerate}
where $f\in\smooth{\cM}$ and $X\in\XX(\cM)$ are homogeneous elements, and $s\in \sections{\cE}$.

The \bfemph{covariant derivative} associated to a
linear connection $\nabla$ is the $\KK$-linear map
\[\coder: \sections{\Lambda^p \tangent{\cM}^{\vee}\otimes \cE} \to
\sections{\Lambda^{p+1} \tangent{\cM}^{\vee}\otimes \cE}\]
of (internal) degree $0$, defined by
\begin{align*}
\left( \coder\omega\right) (X_1\wedge \cdots \wedge X_{p+1})
&=\sum_{i=1}^{p+1} (-1)^{i+1}\sign\cdot \nabla_{X_{i}}
\big(\omega(X_1\wedge \cdots \wedge\widehat{X}_i\wedge\cdots \wedge X_{p+1})\big)\\
&\quad +\sum_{i<j} (-1)^{i+j}\sign\cdot \omega([X_i,X_j]\wedge X_1\wedge
\cdots\wedge\widehat{X}_i\wedge\cdots\wedge\widehat{X}_j\wedge\cdots \wedge X_{p+1}),
\end{align*}
for all homogeneous
$\omega\in\sections{\Lambda^p\tangent{\cM}^{\vee}\otimes\cE}$
and $X_1,\cdots,X_{p+1}\in\XX(\cM)$.
The symbol $\sign=\sign(\omega,X_{1},\cdots,X_{p+1})$
denotes the Koszul signs arising from the reordering of the homogeneous
objects $\omega,X_{1},\cdots,X_{p+1}$ in each term of the right hand side.

We say $\nabla$ is an \bfemph{affine connection} on $\cM$
if it is a linear connection on $\tangent{\cM}$ over $\cM$.
Given an affine connection $\nabla$ on $\cM$,
the $(1,2)$-tensor $T^{\nabla}\in\sections{\tangent{\cM}^{\vee}
\otimes\tangent{\cM}^{\vee}\otimes \tangent{\cM}}$ of degree 0, defined by
\[T^{\nabla}(X,Y)=\nabla_X Y -(-1)^{\degree X \cdot \degree Y} \nabla_Y X - [X,Y] \]
for any homogeneous vector fields $X,Y \in \XX(\cM)$, is called the \bfemph{torsion} of $\nabla$.
We say an affine connection $\nabla$ is \bfemph{torsion-free} if
$T^{\nabla}= 0$.
It is well known that affine torsion-free connections
always exist \cite{MR3910470}.

The \bfemph{curvature} of an affine connection $\nabla$ is
the $(1,3)$-tensor $R^{\nabla}\in
\Omega^2\left({\cM},\End(\tangent{\cM})\right)$ of degree 0, defined by
\[R^{\nabla}(X,Y)Z = \nabla_X \nabla_Y Z -(-1)^{\degree X \cdot \degree Y} 
\nabla_Y \nabla_X Z - \nabla_{[X,Y]} Z\]
for any homogeneous vector fields $X,Y,Z\in\XX(\cM)$.
If the curvature $\curvature$ vanishes identically,
the affine connection $\nabla$ is called \bfemph{flat}.

Let $(\cM,Q)$ be a dg manifold.
We define an operator $\cQ$ of degree +1 on the graded $\smooth{\cM}$-module
$\sections{\cM;\tangent{\cM}^{\vee}\otimes\End(\tangent{\cM})}$:
\begin{equation}\label{eq:cQ}
\cQ:\sections{\cM;\tangent{\cM}^{\vee}\otimes\End(\tangent{\cM})}^{\bullet}
\to\sections{\cM;\tangent{\cM}^{\vee}\otimes\End(\tangent{\cM})}^{\bullet+1}
\end{equation}
by the Lie derivative along the homological vector field $Q$:
\[ (\cQ F)(X,Y)=[Q,F(X,Y)]-(-1)^{k}F([Q,X],Y)-(-1)^{k+\degree X}F(X,[Q,Y]) \]
for any section
$F\in\sections{\cM;\tangent{\cM}^{\vee}\otimes\End(\tangent{\cM})}^{k}$
of degree $k$ and homogeneous vector fields
$X,Y\in\XX(\cM)$. One can easily check that $\cQ^{2}=0$. Therefore
\[ \big(\sections{\cM;\tangent{\cM}^{\vee}\otimes\End(\tangent{\cM})}^{\bullet},
\cQ\big) \]
is a cochain complex.

Now given an affine connection $\nabla$, consider the
$(1,2)$-tensor $\atiyahcocycleQ\in\sections{\cM;\tangent{\cM}^{\vee}
\otimes\End(\tangent{\cM})}$ of degree +1, defined by
\[\atiyahcocycleQ (X,Y) = [Q, \nabla_{X}Y] - \nabla_{[Q,X]}Y
-(-1)^{\degree X} \nabla_{X}[Q,Y] \]
for any homogeneous vector fields $X,Y\in \XX(\cM)$.

\begin{proposition}[\cite{MR3319134}]\label{Atiyahproperty}
In the above setting, the following statements hold.
\begin{enumerate}
\item If the affine connection $\nabla$ on $\cM$ is torsion-free, then
$\atiyahcocycleQ\in\sections{\cM;S^{2}(\tangent{\cM}^{\vee})\otimes\tangent{\cM}}$.
In other words,
\[ \atiyahcocycleQ(X,Y)=(-1)^{\degree X\cdot\degree Y}\atiyahcocycleQ(Y,X) .\]
\item The element $\atiyahcocycleQ\in\sections{\cM;\tangent{\cM}^{\vee}
\otimes\End(\tangent{\cM})}^{1}$ is a $1$-cocycle.
\item The cohomology class $[\atiyahcocycleQ]$ does not depend on the choice of connection.
\end{enumerate}
\end{proposition}

The element $\atiyahcocycleQ$ is called the \bfemph{Atiyah cocycle}
associated with the affine connection $\nabla$.
The cohomology class $\atiyahclassQ:=[\atiyahcocycleQ]\in\cohomology{1}
\big(\sections{\cM;\tangent{\cM}^{\vee}\otimes\End(\tangent{\cM})}^{\bullet},\cQ\big)$
is called the \bfemph{Atiyah class} of the dg manifold $(\cM,Q)$ \cite{MR3319134}.
See also \cite{arXiv:math/9812009} and \cite[Footnote~6]{MR2608525}.

\section{Formal exponential map of dg manifolds}	

\subsection{dg coalgebras}

\subsubsection{dg coalgebras}
Let $\cR$ be a graded commutative ring.
A \bfemph{graded coalgebra} $C$ over $\cR$ is a graded $\cR$-module
equipped with an $\cR$-linear map $\Delta:C\to C\otimes_{\cR}C$ of degree $0$
called comultiplication satisfying the following conditions:
\begin{enumerate}
\item (Coassociativity)
\[(\Delta\otimes \id_{C}) \circ \Delta= (\id_{C}\otimes \Delta)
\circ \Delta : C\to C\otimes_{\cR}C\otimes_{\cR}C.\]
\item (Counit) There is an $\cR$-linear map $\epsilon:C\to \cR$ of degree $0$ such that
\[(\epsilon\otimes \id) \circ \Delta= (\id \otimes \epsilon)\circ \Delta=\id_{C}.\]
\end{enumerate}

Let $\twist: C\otimes_{\cR} C \to C\otimes_{\cR} C$ be the
map defined by
\[ \twist(c_1\otimes c_2)=(-1)^{\degree{c_1}\cdot \degree{c_2}}c_2\otimes c_1, \]
for homogeneous elements $c_1, \ c_2\in C$. A graded coalgebra $C$ is called
\bfemph{cocommutative} if it satisfies $\Delta=\twist\circ \Delta$.

An $\cR$-linear map $\phi :C\to C$ satisfying
\[\Delta \circ \phi = (\id_{C}\otimes \phi+ \phi\otimes \id_{C})\circ \Delta\]
is called an \bfemph{$\cR$-coderivation} of the graded $\cR$-coalgebra $C$.
We denote the collection of all $\cR$-coderivations of $C$ by $\coDer_{\cR}(C)$.

Let $(\cR,d_{\cR})$ be a dg commutative ring,
and $(C,d_{C})$ be a dg $(\cR,d_{\cR})$-module. Then the map
\[ d_{C^{\otimes 2}} : C\otimes_{\cR}C \to C\otimes_{\cR}C \]
defined by
\[ d_{C^{\otimes 2}}(c_1\otimes c_2) = d_{C}(c_1)\otimes c_2 
+ (-1)^{\degree{c_1}}c_1\otimes d_{C}(c_2) \]
for homogeneous elements $c_1,c_2\in C$, is a well-defined degree $+1$ differential.
Such a differential is called the induced differential on $C\otimes_{\cR}C$.

\begin{definition}
Let $(\cR,d_{\cR})$ be a dg commutative ring.
A \bfemph{dg coalgebra} $(C,d_{C})$ over $(\cR,d_{\cR})$ is a dg
$(\cR,d_{\cR})$-module $(C,d_{C})$,
equipped with a graded coalgebra structure on
$C$ over $\cR$ where the comultiplication and the counit map respect the differentials. That is,
\begin{gather*}
\Delta \circ d_{C} = d_{C^{\otimes 2}}\circ \Delta ,\\
\epsilon \circ d_{C}=d_{\cR}\circ \epsilon
\end{gather*}
where $\Delta: C\to C\otimes_{\cR}C$ is the comultiplication
and $\epsilon:C\to \cR$ is the counit map.
\end{definition}

\subsubsection{dg coalgebras associated to dg manifolds}
Any dg manifold $(\cM,Q)$ determines a pair 
of dg coalgebras over the dg ring
$(\smooth{\cM},Q)$, namely 
$\cD(\cM)$ and $\sections{S(\tangent{\cM})}$.
Below we will briefly describe these dg coalgebra structures.
In the sequel, unless specified otherwise,
we will always identify $(\cR,d_{\cR})\cong(\smooth{\cM},Q)$.

First, let us consider the dg coalgebra structure on the 
left $\cR$-module $\cD(\cM)$ of differential operators on $\cM$.

The comultiplication
\begin{equation}\label{eq:1}
\Delta:\cD(\cM)\to\cD(\cM)\otimes_{\cR}\cD(\cM)
\end{equation}
is defined by
\[(\Delta D)(f\otimes g) = D(f\cdot g), \]
where $f,g\in\smooth{\cM}$ and $D\in\cD(\cM)$.

The differential $\UdifferentialM : \cD(\cM)\to \cD(\cM)$
is defined as the commutator with $Q$, which is also the Lie derivative
along the homological vector field $Q$:
\begin{equation}\label{eq:2}
\UdifferentialM(D)
=\llbracket Q,D\rrbracket
=Q\cdot D-(-1)^{\degree{D}}D\cdot Q
\end{equation}
for any $D\in\cD(\cM)$, where $\llbracket\argument,\argument\rrbracket$
denotes the commutator on $\cD(\cM)$.

The induced differential on
$\cD(\cM)\otimes_{\cR}\cD(\cM)$ is again the Lie derivative
$\UdifferentialM$, which coincides with
$\llbracket Q,\argument\rrbracket$, with $\llbracket\argument,\argument\rrbracket$
being the Gerstenhaber bracket on polydifferential operators
on $\cM$.

The counit map
\begin{equation}\label{eq:2-1}
\epsilon:\cD(\cM)\to \smooth{\cM}
\end{equation}
is the canonical projection, which evaluates a differential operator $D$
on the constant function $1$.

Note that $\cD(\cM)$ admits a natural ascending filtration by the order of differential operators
\[ \smooth{\cM}=\cD^{\leq 0}(\cM) \subset \cdots \subset \cD^{\leq n}(\cM)\subset \cdots\]
where $\cD^{\leq n}(\cM)$ denotes the space of differential operators of order $\leq n$.
The following proposition can be easily verified.

\begin{proposition}\label{pro:2.2}
For any dg manifold $(\cM, Q)$, the space of differential operators
$\cD(\cM)$ on $\cM$, equipped with the comultiplication
$\Delta$, the differential $\UdifferentialM$ and the counit $\epsilon$ as in~\eqref{eq:1},
\eqref{eq:2} and~\eqref{eq:2-1},
is a filtered dg cocommutative coalgebra over
$(\smooth{\cM},Q)$.
\end{proposition}

Next we describe the dg coalgebra structure on the left $\cR$-module $\sections{S(\tangent{\cM})}$.

The comultiplication
\begin{equation*}
\Delta:\sections{S(\tangent{\cM})}\to\sections{S(\tangent{\cM})}
\otimes_{\cR}\sections{S(\tangent{\cM})}
\end{equation*}
is given by
\begin{align}
\Delta(X_1\odot \cdots \odot X_n)&=1\otimes (X_1\odot \cdots\odot X_n)
+(X_1\odot \cdots \odot X_n) \otimes 1 \nonumber \\
&+ \sum_{k=1}^{n-1}\sum_{\sigma \in \shuffle{k}{n-k}} 
\sign\cdot (X_{\sigma(1)}\odot \cdots \odot X_{\sigma(k)})
\otimes (X_{\sigma(k+1)}\odot \cdots \odot X_{\sigma(n)}) \label{eq:3}
,\end{align}
where $X_1,\cdots,X_n\in\sections{\tangent{\cM}}$.
The symbol $\shuffle{k}{n-k}$ denotes the 
set of all $(k,n-k)$-shuffles and the symbol 
$\sign:=\sign(X_1,X_2,\cdots,X_n)$
denotes the Koszul signs arising
from the reordering of the homogeneous
objects $X_1,X_2,\cdots,X_n$ in each term
of the right hand side.

The differential
\begin{equation}\label{eq:4}
\SdifferentialM:\sections{S(\tangent{\cM})}
\to\sections{S(\tangent{\cM})}
\end{equation}
is the Lie derivative along the homological vector field $Q$.
The induced differential on
$\sections{S(\tangent{\cM})}\otimes_{\cR}\sections{S(\tangent{\cM})}
\cong \sections{S(\tangent{\cM})\otimes S(\tangent{\cM})}$
is again the Lie derivative $\liederivative{Q}$.

The counit map
\begin{equation}\label{eq:4-2}
\epsilon:\sections{S(\tangent{\cM})} \to \smooth{\cM}
\end{equation}
is the canonical projection.

Note that $\sections{S(\tangent{\cM})}$ admits a canonical ascending filtration
\[ \smooth{\cM}=\sections{S^{\leq 0}(\tangent{\cM})}\subset
\cdots\subset\sections{S^{\leq n}(\tangent{\cM})}\subset\cdots. \]
The following proposition is easily verified.
\begin{proposition}\label{pro:2.3}
For any dg manifold $(\cM, Q)$,
the space $\sections{S(\tangent{\cM})}$,
equipped with the comultiplication $\Delta$,
the differential $\SdifferentialM$ and the counit map $\epsilon$ as in~\eqref{eq:3},
\eqref{eq:4} and~\eqref{eq:4-2},
is a filtered dg cocommutative coalgebra over $(\smooth{\cM},Q)$.
\end{proposition}

\subsection{Formal exponential map of a dg manifold}

Let $\cM$ be a finite dimensional graded manifold
and $\nabla$ be an affine connection on $\cM$.
A purely algebraic description of the Poincar\'e--Birkhoff--Witt map
has been extended to the context of $\ZZ$-\emph{graded} manifolds by
Liao--Sti\'enon \cite{MR3910470}. As pointed out in the introduction,
for an ordinary smooth manifold, the PBW map is
a formal exponential map.
In the same way, one can think of the
PBW map of a $\ZZ$-graded manifold as an
induced formal exponential map of `the virtual exponential map'
\begin{equation}\label{eq:expcm}
\exp^{\nabla}:\tangent{\cM}\to\cM\times\cM
\end{equation}
by taking fiberwise $\infty$-jets.

Recall that the Poincar\'e--Birkhoff--Witt map
\begin{equation}\label{defnpbw}
\pbw^{\nabla}:\sections{S(\tangent{\cM})}\to\cD(\cM)
\end{equation}
is defined by the inductive formula \cite{MR3910470}:
\begin{equation}
\begin{gathered}\label{eqnpbw}
\pbw^{\nabla}(f)=f, \quad\forall f\in\smooth{\cM};\\
\pbw^{\nabla}(X)=X, \quad \forall X\in \XX(\cM);
\end{gathered}
\end{equation}
and
\begin{equation}\label{formulaofpbw}
\pbw^{\nabla}(X_1\odot \cdots \odot X_n) = \frac{1}{n} \sum_{k=1}^{n} \sign_k 
\left ( X_k \pbw^{\nabla}(\boldX^{\{k\}}) - \pbw^{\nabla}(\nabla_{X_k} \boldX^{\{k\}}) \right)
,\end{equation}
where $\boldX=X_1\odot\cdots\odot X_n\in\sections{S^{n}(\tangent{\cM})}$
for homogeneous vector fields $X_1,\cdots,X_n \in \XX(\cM)$
and $\sign_{k}=(-1)^{\degree{X_k}(\degree{X_1}+\cdots+\degree{X_{k-1}})}$
is the Koszul sign.

\begin{theorem}[\cite{MR3910470}]\label{thm:pbw}
The map $\pbw^{\nabla}$ is an isomorphism of graded coalgebras
from $\sections{S(\tangent{\cM})}$ to $\cD(\cM)$ over $\smooth{\cM}$.
\end{theorem}

Now, we assume there exists a homological vector field $Q$ on $\cM$
so that $(\cM,Q)$ is a dg manifold. Then,
both $\Gamma(S(\tangent{\cM}))$ and $\cD(\cM)$ in~\eqref{defnpbw}
are dg coalgebras over $(\smooth{\cM},Q)$,
according to Propositions~\ref{pro:2.2} and~\ref{pro:2.3}.
We think of the elements of the dg coalgebra $\big(\sections{S(T_\cM)},\SdifferentialM\big)$ 
as fiberwise dg distributions on the dg vector bundle $\pi: T_\cM\to \cM$
with support the zero section --- the homological vector field on $T_\cM$ is 
the complete lift $\hat{Q}$ of the homological vector field $Q$ on $\cM$
\cite{MR3319134,MR4276044}. 
Likewise, we think of the elements of the dg coalgebra
$\big(\cD(\cM),\UdifferentialM\big)$
as fiberwise dg distributions on the dg fiber bundle
$\pr_1: \cM\times \cM\to \cM$ with support the diagonal
$\Delta$ --- the homological vector field on $\cM\times\cM$ is $(Q,Q)$.
On the level of fiberwise $\infty$-jets, the fact that
the virtual exponential map \eqref{eq:expcm}
is a map of dg manifolds
is equivalent
to the map $\pbw^{\nabla}:\big(\sections{S(\tangent{\cM})},\SdifferentialM\big)
\to\big(\cD(\cM),\UdifferentialM\big)$ being an isomorphism of
dg coalgebras over $(\smooth{\cM},Q)$.
This consideration leads to the following
\begin{theorem}\label{theorem1}
Let $(\cM,Q)$ be a dg manifold.
The Atiyah class $\atiyahclassQ$ vanishes
if and only if there exists a torsion-free affine connection $\nabla$ on $\cM$
such that \[ \pbw^\nabla:\big(\sections{S(\tangent{\cM})},\SdifferentialM\big)
\to\big(\cD(\cM),\UdifferentialM\big) \]
is an isomorphism of dg coalgebras over $(C^\infty(\cM),Q)$.
\end{theorem}

\begin{remark}
A similar theorem in the same spirit
concerning the Atiyah class of Lie pairs
was obtained in \cite[Theorem~5.10]{MR4271478}.
It would be interesting to establish a result
that encompasses both \cite[Theorem~5.10]{MR4271478}
and Theorem~\ref{theorem1} under a unified framework.
\end{remark}

In order to prove
Theorem~\ref{theorem1}, we first introduce a linear map
\[ \pbwcommutator:\sections{S(\tangent{\cM})}\to\cD(\cM) \]
by
\begin{equation}\label{eq:C}
\pbwcommutator:=\UdifferentialM\circ\pbw^{\nabla}-\pbw^{\nabla}\circ\SdifferentialM.
\end{equation}
One can easily check that $\pbwcommutator$ is a $C^\infty(\cM)$-linear map of degree +1.
Moreover, for $n\geq 0$,
\[ \pbwcommutator\big(\sections{S^{\leq n}(\tangent{\cM})}\big)
\subseteq\cD^{\leq n-1}(\cM) .\]

The following proposition indicates that $\pbwcommutator$
can be completely determined by a recursive formula.

\begin{proposition}\label{formula1}
Let $(\cM,Q)$ be a dg manifold, and $\nabla$ a torsion-free
affine connection on $\cM$.
Then the map $\pbwcommutator$ satisfies
\begin{gather}
\pbwcommutator(f)=0; \label{formulaA1} \\
\pbwcommutator(X)=0; \label{formulaA2} \\
\pbwcommutator(X\odot Y)=-\atiyahcocycleQ(X,Y), \label{formulaA4}
\end{gather}
for all $f\in\smooth{\cM}$, $X,Y\in\XX(\cM)$, and,
for $n\geq3$, it satisfies the recursive formula
\begin{multline}\label{formulaA3}
\pbwcommutator(\boldX)=\frac{1}{n}\sum_{k=1}^{n} \sign_{k}
\left [ (-1)^{\degree{X_{k}}}X_{k}\cdot \pbwcommutator(\boldX^{\{k\}})
-\pbwcommutator(\nabla_{X_k}\boldX^{\{k\}})\right] \\
-\frac{2}{n} \sum_{i<j} \sign_{i} \sign_{j} (-1)^{\degree{X_{i}}
\cdot\degree{X_{j}}} \pbw^{\nabla}\left( \atiyahcocycleQ(X_i, X_j)
\odot \boldX^{\{i,j\}}\right),
\end{multline}
where $\boldX=X_{1}\odot\cdots\odot X_{n}\in\sections{S^{n}(\tangent\cM)}$
denotes the symmetric tensor product of 
$n$ homogeneous vector fields
$X_1,\cdots,X_n \in \XX(\cM)$;
$\boldX^{\{k\}}=X_{1}\odot \cdots\odot \widehat{X}_{k} \odot\cdots \odot X_{n}$ for
any $1\leq k\leq n$; 
$\boldX^{\{i,j\}}=X_{1}\odot\cdots\odot \widehat{X}_{i}\odot\cdots\odot 
\widehat{X}_{j}\odot\cdots\odot X_{n}$ for
any $1\leq i < j\leq n$; and
$\sign_{k}=(-1)^{\degree{X_k}(\degree{X_1}+\cdots+\degree{X_{k-1}})}$ is the Koszul sign
arising from the reordering 
$X_1,X_2,\cdots,X_n\mapsto
X_k,X_1,X_2,\cdots,X_{k-1},X_{k+1},\cdots,X_n$.
\end{proposition}

We now prove Theorem~\ref{theorem1} based on
Proposition~\ref{formula1}.

\begin{proof}[Proof of Theorem~\ref{theorem1}]
Observe that, according to Proposition~\ref{Atiyahproperty},
$\atiyahclassQ$ vanishes if and only if there exists an
affine connection $\nabla'$ for which $\At^{\nabla'}_{(\cM,Q)}=0$.
It follows from $\liederivative{Q}(\nabla')=\At^{\nabla'}_{(\cM,Q)}=0$
that $\liederivative{Q}(T^{\nabla'})=0$.
Therefore, if the Atiyah cocycle of the affine connection $\nabla'$ vanishes, 
then so does the Atiyah cocycle of the torsion-free connection 
$\nabla=\nabla'-\frac{1}{2}T^{\nabla'}$:
\[ \atiyahcocycleQ=\liederivative{Q}(\nabla)
=\liederivative{Q}(\nabla'-\frac{1}{2}T^{\nabla'})
=\At^{\nabla'}_{(\cM,Q)}-\frac{1}{2}\liederivative{Q}(T^{\nabla'})=0 .\]

Thus, it suffices to prove
that $\pbwcommutator = 0$ if and only if $\atiyahcocycleQ = 0$.

Assume that $\pbwcommutator = 0$. By Proposition~\ref{formula1}, we have
\[ \atiyahcocycleQ(X,Y)
=-\pbwcommutator(X\odot Y)=0 \]
for all $X,Y\in\XX(\cM)$.

Conversely, assume that $\atiyahcocycleQ=0$. Then we have
$\pbwcommutator(X\odot Y)=0$ by Proposition~\ref{formula1}.
Hence $\pbwcommutator(\boldY)=0$ for all $\boldY\in
\sections{S^{\leq 2}(\tangent{\cM})}$.
Moreover, Equation~\eqref{formulaA3} can be written as
\[\pbwcommutator(\boldX)=\frac{1}{n}\sum_{k=1}^{n} \sign_{k}\left
[ (-1)^{\degree{X_{k}}}X_{k}\cdot \pbwcommutator(\boldX^{\{k\}})
-\pbwcommutator(\nabla_{X_k}\boldX^{\{k\}})\right], \ \ \ \forall \boldX \in
\big(\sections{S^{\geq 3}(\tangent{\cM})}\big)
.\]
Therefore, $\pbwcommutator=0$ by the inductive argument.
\end{proof}

\subsection{Proof of Proposition~\ref{formula1}}
Now we turn to the proof of Proposition~\ref{formula1}.
We will divide the proof into several lemmas.

\begin{lemma}\label{lem:1-3}
Under the same hypothesis as in Proposition~\ref{formula1},
Equations~\eqref{formulaA1}, \eqref{formulaA2} and~\eqref{formulaA4} hold.
\end{lemma}
\begin{proof}

Equations~\eqref{formulaA1} and~\eqref{formulaA2} follow immediately
from Equation~\eqref{eqnpbw}.

To prove Equation~\eqref{formulaA4},
let $X,Y\in \XX(\cM)$ be homogeneous vector fields. Since $\nabla$ is
torsion-free, we have
\begin{equation*}
\nabla_{X}Y-(-1)^{\degree{X}\cdot \degree{Y}}\nabla_{Y}X=[X,Y]=XY-YX .
\end{equation*}
It then follows from Equation~\eqref{formulaofpbw} that
\[\pbw^{\nabla}(X\odot Y)=XY - \nabla_{X}Y.\]

From there, we obtain
\begin{equation*}
\UdifferentialM \circ \pbw^{\nabla}(X\odot Y)=[Q,X]Y+(-1)^{\degree{X}}X[Q,Y]-[Q,\nabla_{X}Y]
\end{equation*}
and
\begin{align*}
\pbw^{\nabla}\circ \SdifferentialM(X\odot Y)&=\pbw^{\nabla}\left([Q,X]\odot Y 
+ (-1)^{\degree{X}} X\odot [Q,Y] \right)\\
&=\left( [Q,X]Y - \nabla_{[Q,X]}Y\right)
+(-1)^{\degree{X}}\left( X[Q,Y]- \nabla_{X}[Q,Y]\right).
\end{align*}
As a result, we have
\begin{align*}
\pbwcommutator(X\odot Y) &=(\UdifferentialM \circ \pbw^{\nabla}
-\pbw^{\nabla}\circ \SdifferentialM)(X\odot Y)\\
&= -\left( [Q,\nabla_{X}Y] - \nabla_{[Q,X]}Y-(-1)^{\degree{X}} \nabla_{X}[Q,Y]\right) \\
&=-\atiyahcocycleQ(X,Y).
\qedhere\end{align*}
\end{proof}

In the sequel, we adopt the following notations.
For any $\boldX=X_{1}\odot \cdots \odot X_{n}\in
\sections{S^{n}(\tangent{\cM})}$, we write
$\boldX^{\{k\}}=X_{1}\odot\cdots\odot\widehat{X}_{k}\odot\cdots\odot X_{n}$; for
$i\neq j$, we write 
$\boldX^{\{i,j\}}=X_{1}\odot\cdots\odot
\widehat{X}_{i}\odot\cdots\odot
\widehat{X}_{j}\odot\cdots\odot X_{n}$,
and for all $1\leq i\leq n$,
$\boldX^{\{i,i\}}=0$.

\begin{lemma}\label{lem:4}
Under the same hypothesis as in Proposition~\ref{formula1}, for all
$\boldX=X_1\odot\cdots\odot X_n\in\sections{S^{n}(\tangent{\cM})}$
with $n\geq 3$, we have
\begin{align*}
\UdifferentialM\circ \pbw^{\nabla}(\boldX)&=\frac{1}{n}\sum_{k=1}^{n}\varepsilon
\cdot [Q,X_k]\cdot \pbw^{\nabla}(\boldX^{\{k\}}) \\
& \quad +\frac{1}{n}\sum_{k=1}^{n} \left(\varepsilon\cdot X_k\cdot
\UdifferentialM\big( \pbw^{\nabla}(\boldX^{\{k\}})\big)
-\varepsilon\cdot \UdifferentialM\big( \pbw^{\nabla}(\nabla_{X_k}\boldX^{\{k\}})\big)\right)
\end{align*}
and
\begin{align*}
&\pbw^{\nabla} \circ \SdifferentialM (\boldX)\\
&=\frac{1}{n}\sum_{k=1}^{n}\left(\varepsilon\cdot [Q,X_k]\cdot\pbw^{\nabla}(\boldX^{\{k\}})
+\varepsilon\cdot X_k\cdot \pbw^{\nabla}\big(\SdifferentialM(\boldX^{\{k\}})\big)
-\varepsilon\cdot \pbw^{\nabla}\big(\SdifferentialM(\nabla_{X_k}\boldX^{\{k\}})\big)\right)\\
&\quad +\frac{1}{n}\sum_{i<j}\varepsilon\cdot
\pbw^{\nabla}\left(2\atiyahcocycleQ(X_i,X_j)\odot \boldX^{\{i,j\}}\right).
\end{align*}
In the two equations above and in the proof of the Lemma as well,
the symbol $\sign=\sign(Q,X_1,\cdots,X_n)$ denotes the Koszul signs
arising from the reordering of the homogeneous objects
$Q,X_1,\cdots,X_n$ in each term of the right hand sides.
\end{lemma}

\begin{proof}
The formula for $\UdifferentialM\circ \pbw^{\nabla}(\boldX)$ is immediate
from Equation~\eqref{formulaofpbw}.

Next, we will compute $\pbw^{\nabla}\circ \SdifferentialM(\boldX)$.
Since
$\SdifferentialM (\boldX)= \sum_{k=1}^{n} \varepsilon \cdot ([Q, X_k] \odot \boldX^{\{k\}})$,
applying Equation~\eqref{formulaofpbw}, we have
\[ \pbw^{\nabla} \circ \SdifferentialM(\boldX)=
\frac{1}{n}\left( \cA^{1}-\cA^{2}+\cB-\cC \right) ,\]
where
\begin{align}
\cA^{1}&:= \sum_{k=1}^{n} \varepsilon\cdot [Q,X_k]
\cdot \pbw^{\nabla}(\boldX^{\{k\}}), \label{eq:ROM}\\
\cA^{2}&:= \sum_{k=1}^{n} \varepsilon\cdot
\pbw^{\nabla}(\nabla_{[Q,X_k]} \boldX^{\{k\}}), \nonumber\\
\cB&:= \sum_{k=1}^{n} \sum_{i=1}^n \varepsilon\cdot X_i
\cdot \pbw^{\nabla}([Q,X_k]\odot \boldX^{\{i,k\}}),\nonumber\\
\cC&:= \sum_{k=1}^{n} \sum_{i=1}^{n} \varepsilon
\cdot\pbw^{\nabla}(\nabla_{X_i}([Q,X_k]\odot \boldX^{\{i,k\}}) . \nonumber
\end{align}

First, by changing the order of summation, we obtain
\begin{align}
\cB &=\sum_{i=1}^{n} \sum_{k=1}^{n} \varepsilon \cdot X_i 
\cdot \pbw^{\nabla}\left ([Q,X_k]\odot \boldX^{\{i,k\}}\right) \nonumber\\
& = \sum_{i=1}^{n} \varepsilon\cdot X_i\cdot 
\pbw^{\nabla}\left (\SdifferentialM(\boldX^{\{i\}}) \right). \label{eqnB}
\end{align}

We also can write
\begin{align}
\cA^{2} &= \sum_{k=1}^{n} \sum_{i=1}^n \varepsilon \cdot 
\pbw^{\nabla}\left((\nabla_{[Q,X_k]}X_i)\odot \boldX^{\{k,i\}}\right)\nonumber \\
&= \sum_{k=1}^n \sum_{i=1}^{n} \varepsilon\cdot 
\pbw^{\nabla}\left((\nabla_{[Q,X_i]}X_k)\odot \boldX^{\{i,k\}}\right). \label{eq:SCE}
\end{align}

Now we also have
\begin{align*}
&\sum_{k=1}^{n} \sum_{i=1}^{n}\varepsilon \cdot
\pbw^{\nabla}\big( [Q,X_k]\odot \nabla_{X_i}\boldX^{\{i,k\}} \big)\\
=&\sum_{k=1}^{n} \sum_{i=1}^{n} \sum_{j=1}^{n}\varepsilon \cdot
\pbw^{\nabla}\big([Q,X_k]\odot \nabla_{X_i} X_j \odot \boldX^{\{i,k,j\}} \big)\\
=&\sum_{k=1}^{n} \sum_{i=1}^{n} \sum_{j=1}^{n}\varepsilon \cdot
\pbw^{\nabla}\big( \nabla_{X_i} X_j \odot [Q,X_k] \odot \boldX^{\{i,k,j\}} \big)\\
=&\sum_{i=1}^{n} \sum_{j=1}^{n} \varepsilon \cdot
\pbw^{\nabla}\big( \nabla_{X_i} X_j \odot
\SdifferentialM \boldX^{\{i,j\}} \big)\\
=&\sum_{i=1}^{n} \sum_{k=1}^{n} \varepsilon \cdot
\pbw^{\nabla}\big( \nabla_{X_i} X_k \odot \SdifferentialM \boldX^{\{i,k\}} \big).
\end{align*}

Therefore, it follows that
\begin{align}
&\sum_{k=1}^{n} \sum_{i=1}^{n}\varepsilon\cdot \pbw^{\nabla}
\big([Q,\nabla_{X_i}X_k]\odot \boldX^{\{i,k\}}\big)
+\sum_{k=1}^{n} \sum_{i=1}^{n}\varepsilon \cdot \pbw^{\nabla}
\big( [Q,X_k]\odot \nabla_{X_i}\boldX^{\{i,k\}} \big) \nonumber \\
=&\sum_{k=1}^{n} \sum_{i=1}^{n}\varepsilon\cdot \pbw^{\nabla}
\big([Q,\nabla_{X_i}X_k]\odot \boldX^{\{i,k\}}\big)
+\sum_{i=1}^{n} \sum_{k=1}^{n} \varepsilon \cdot
\pbw^{\nabla}\big( \nabla_{X_i} X_k \odot \SdifferentialM \boldX^{\{i,k\}} \big) \nonumber \\
=&\sum_{i=1}^{n} \sum_{k=1}^{n} \varepsilon\cdot \pbw^{\nabla} \SdifferentialM
\big(\nabla_{X_i}X_k \odot \boldX^{\{i,k\}}\big) \nonumber \\
=&\sum_{i=1}^{n} \varepsilon\cdot \pbw^{\nabla} 
\big( \SdifferentialM (\nabla_{X_i}X^{\{i\}} )\big). \label{eq:CDG}
\end{align}

Moreover,
\begin{equation}\label{eqnC1}
\cC=\sum_{k=1}^{n} \sum_{i=1}^{n} \varepsilon\cdot 
\pbw^{\nabla}\big( (\nabla_{X_i}[Q,X_k])\odot \boldX^{\{i,k\}}\big) +
\sum_{k=1}^{n} \sum_{i=1}^{n} \varepsilon\cdot \pbw^{\nabla}
\big( [Q,X_k]\odot \nabla_{X_i}\boldX^{\{i,k\}} \big)
.\end{equation}

Then by combining Equations~\eqref{eq:SCE}, \eqref{eq:CDG} and~\eqref{eqnC1}
and using the definition of Atiyah cocycles, we obtain
\begin{align}
\cA^{2}+\cC
&=\sum_{k=1}^{n} \sum_{i=1}^{n}\varepsilon\cdot 
\pbw^{\nabla}\big(\big([Q,\nabla_{X_i}X_k]-\atiyahcocycleQ(X_i,X_k)\big)
\odot \boldX^{\{i,k\}}\big) \nonumber \\
&\quad +\sum_{k=1}^{n} \sum_{i=1}^{n} \varepsilon\cdot \pbw^{\nabla} 
\big( [Q,X_k]\odot \nabla_{X_i}\boldX^{\{i,k\}} \big) \nonumber \\
&=\sum_{i=1}^{n} \varepsilon\cdot \pbw^{\nabla} \big (\SdifferentialM
(\nabla_{X_i}X^{\{i\}} )\big)
-\sum_{i<j}\varepsilon\cdot \pbw^{\nabla}
\left(2\atiyahcocycleQ(X_i,X_j)\odot \boldX^{\{i,j\}}\right). \label{eq:NAP}
\end{align}

The conclusion thus follows from Equations~\eqref{eq:ROM}, \eqref{eqnB}, and~\eqref{eq:NAP}.
\end{proof}

\begin{proof}[Proof of Proposition~\ref{formula1}]
Equations~\eqref{formulaA1}, \eqref{formulaA2} and~\eqref{formulaA4} have been proved in
Lemma~\ref{lem:1-3}. It remains to prove Equation~\eqref{formulaA3}.
According to Lemma~\ref{lem:4}, we have
\begin{align*}
\UdifferentialM\circ \pbw^{\nabla}(\boldX) - \pbw^{\nabla}\circ \SdifferentialM(\boldX)
&=\frac{1}{n}\sum_{k=1}^{n}\sign_{k} (-1)^{\degree{X_k}}X_k\cdot (\UdifferentialM\circ 
\pbw^{\nabla}-\pbw^{\nabla}\circ\SdifferentialM)(\boldX^{\{k\}})\\
&\quad -\frac{1}{n}\sum_{k=1}^{n}\sign_{k}
\big( \UdifferentialM\circ \pbw^{\nabla}-\pbw^{\nabla}\circ\SdifferentialM)
(\nabla_{X_k}\boldX^{\{k\}} \big)\\
&\quad -\frac{1}{n}\sum_{i<j}\sign_{i}\sign_{j}(-1)^{\degree{X_{i}}
\cdot\degree{X_{j}}}\pbw^{\nabla}\left(2\atiyahcocycleQ(X_i,X_j)
\odot \boldX^{\{i,j\}}\right).
\end{align*}
This concludes the proof of Proposition~\ref{formula1}.
\end{proof}

\section{Atiyah class and homotopy Lie algebras}

This section is devoted to the study of homotopy Lie algebras
associated to the Atiyah class of dg manifolds.

\subsection{Kapranov \texorpdfstring{$L_\infty[1]$}{L∞[1]} algebras of dg manifolds}

The Atiyah class of a holomorphic vector bundle is closely related to
$L_\infty[1]$ algebras as shown by the pioneer work
of Kapranov \cite{MR1671737,MR2431634,MR2661534}.
These $L_\infty[1]$ algebras play an important role
in derived geometry~\cite{MR2657369,MR2472137,MR2431634}
and construction of Rozansky--Witten
invariants \cite{MR1671737,MR1671725,MR2661534,arXiv:math/0404360,MR3322372}.

In this section, following Kapranov~\cite{MR1671737},
we show that the Atiyah class of a dg manifold
is related to $L_\infty[1]$ algebras in a similar fashion.
We refer to \cite[Sections~4 and~5]{MR4267551} for the interpretation
in terms of derived category.

Let $(\cM,Q)$ be a dg manifold and let $\nabla$ be an
affine connection on $\cM$.
The Lie derivative $\UdifferentialM$
along the homological vector field $Q$ is a degree~$+1$ coderivation of
the dg coalgebra $\cD(\cM)$ over $(C^\infty(\cM),Q)$ according to
Proposition~\ref{pro:2.2}.

Transferring $\UdifferentialM$ from
$\cD(\cM)$ to $\sections{S(T_\cM)}$ by
the graded coalgebra isomorphism $\pbw^\nabla$ \eqref{defnpbw},
we obtain a degree~$+1$
coderivation $\delta^\nabla$ of $\sections{S(T_\cM)}$:
\begin{equation}\label{eq:delta}
\delta^\nabla:=(\pbw^\nabla)^{-1}\circ \UdifferentialM \circ\pbw^\nabla.
\end{equation}

Therefore
\begin{equation}\label{eq:YVR}
\big(\sections{S(T_\cM)}, \delta^\nabla\big)
\end{equation}
is a dg coalgebra over the dg ring $(C^\infty(\cM),Q)$.

Finally, dualizing $\delta^\nabla$ over $(C^\infty(\cM),Q)$, we obtain
a degree $+1$ derivation:
\begin{equation}\label{eq:DD}
D^\nabla:\sections{\widehat{S}(T^\vee_\cM)}\to\sections{\widehat{S}(T^\vee_\cM)}
\end{equation}
Here we used the identification $\sections{\widehat{S}(T^\vee_\cM)}\cong
\Hom_{C^\infty(\cM)}(\sections{S(T_\cM)},C^\infty(\cM))$.

The following theorem was first announced
in \cite{MR3319134}, but a proof was omitted.
We will present a complete proof below.

\begin{theorem}\label{RSPtheorem}
Let $(\cM,Q)$ be a dg manifold, and let $\nabla$ be a
torsion-free affine connection on $\cM$.
\begin{enumerate}
\item[(i)] The operator $D^\nabla$
is a derivation of degree~$+1$ of the graded algebra $\sections{\widehat{S}(T^\vee_\cM)}$
satisfying $(D^\nabla)^2=0$. Thus $\big( \sections{\widehat{S}(T^\vee_\cM)}, D^\nabla \big)$
is a dg algebra.
\item[(ii)] There exists a sequence of degree $+1$ sections
$R_k\in \sections{S^k(T^\vee_\cM) \otimes T_\cM}$, $k\geqslant 2$
whose first term $R_2$ equals to $-\atiyahcocycleQ$,
such that
\[ D^\nabla=\liederivative{Q}+\sum_{k=2}^\infty\widetilde{R_k} ,\]
where each $\widetilde{R_k}:\sections{\widehat{S}(T^\vee_\cM)}
\to\sections{\widehat{S}(T^\vee_\cM)}$
denotes the $\cR$-linear degree $+1$ derivation corresponding to $R_k$.
\item[(iii)] Different choices of torsion-free affine connections $\nabla$ induce
isomorphic dg algebras $\big( \sections{\widehat{S}(T^\vee_\cM)}, D^\nabla \big)$.
\end{enumerate}
\end{theorem}

\begin{remark}
The graded algebra $\sections{\widehat{S}(\tangent{\cM}^{\vee})}$
can be thought of as the graded algebra of functions
on a graded manifold $\widetilde{T}_{\cM}$ with support $M$
and $D^{\nabla}$ as a homological vector field on $\widetilde{T}_{\cM}$.
Note that $\tangent{\cM}$ and $\widetilde{T}_{\cM}$ are different graded manifolds:
the support of $\tangent{\cM}$ is $\tangent{M}$
while the support of $\widetilde{T}_{\cM}$ is $M$.
\end{remark}

Before we prove this theorem, we need to recall some
basic notations.

Recall that given a graded commutative algebra $\cR$ and a graded $\cR$-module $V$,
the symmetric tensor algebra $(S_{\cR}(V),\mu)$ over $\cR$
admits a canonical graded coalgebra structure
$\Delta:S_{\cR}(V)\to S_{\cR}(V)\otimes_{\cR}S_{\cR}(V)$
defined by \cite{MR4271478}
\begin{align*}
\Delta(v_{1}\odot\cdots \odot v_n)&=1\otimes (v_{1}\odot 
\cdots \odot v_{n})+(v_{1}\odot \cdots \odot v_{n})\otimes 1\\
&\quad +\sum_{k=1}^{n-1}\sum_{\sigma\in \shuffle{k}{n-k}} 
\sign \cdot (v_{\sigma(1)}\odot \cdots \odot v_{\sigma(k)}) 
\otimes (v_{\sigma(k+1)}\odot \cdots \odot v_{\sigma(n)})
\end{align*}
for homogeneous elements $v_1,\cdots, v_{n}\in V$.
Here the symbol $\sign=\sign(v_1,v_2,\cdots,v_n)$
denotes the Koszul signs arising from the reordering of the homogeneous
objects $v_1,v_2,\cdots,v_n$ in each term of the right hand side.

The following lemma is standard--- see, for example, \cite{MR3353026,MR4271478}.
\begin{lemma}\label{coderivation}
Let $\cR$ be a graded commutative algebra and $V$ be an $\cR$-module.
There is a natural isomorphism
\[ \coDer_{\cR}({S}_{\cR}(V), {S}_{\cR}(V))\xrightarrow{\simeq}
\prod_{k=0}^{\infty} \Hom_{\cR}({S}_{\cR}^{k}(V),V)\]
as $\cR$-modules.

More explicitly, the correspondence between
a sequence of maps $\{q_k\}_{k\geq 0}$
with $q_k\in \Hom_{\cR}(S_{\cR}^{k}(V),V)$ and a
coderivation $Q\in \coDer_{\cR}({S}_{\cR}(V),{S}_{\cR}(V))$
is given by
\begin{equation}\label{eq:coder}
\begin{aligned}
Q(v_1\odot \cdots \odot v_n)&=q_{0}(1)\odot v_{1}\odot \cdots \odot v_{n} 
+ q_{n}(v_{1}\odot \cdots \odot v_{n})\odot 1\\
&\quad +\sum_{k=1}^{n-1}\sum_{\sigma\in \shuffle{k}{n-k}}\sign \cdot q_k(v_{\sigma(1)}
\odot\cdots\odot v_{\sigma(k)})\odot v_{\sigma(k+1)}\odot\cdots\odot v_{\sigma(n)}
\end{aligned}
\end{equation}
for homogeneous vectors $v_1,\cdots,v_n\in V$.
\end{lemma}

For a given graded $\cR$-coalgebra $(C, \Delta)$ and a graded $\cR$-algebra
$(A,\mu)$, the convolution product $\star$ on the graded vector space
$\Hom_\cR (C,A)$ is defined by
\[f\star g = \mu \circ (f\otimes g) \circ \Delta\]
$ \forall f,g \in \Hom_{\cR}(C,A)$. It is clear that
$(\Hom_\cR (C,A), \star)$ is a graded $\cR$-algebra.
In particular, since $S_{\cR}(V)$ is both a graded coalgebra and
a graded algebra, the space of $\cR$-linear maps 
$\Hom_{\cR}\left(S_{\cR}(V), S_{\cR}(V) \right)$ admits a convolution product:
\begin{equation}\label{eq:convolutionproduct}
(f\star g)(\mathbf{v})=\sum_{(\mathbf{v})}(-1)^{\degree{g}
\cdot \degree{\mathbf{v}_{(1)}}} f(\mathbf{v}_{(1)})\odot g(\mathbf{v}_{(2)}),
\end{equation}
where $\mathbf{v}\in S_{\cR}(V)$ and $\Delta(\mathbf{v})
=\sum\limits_{(\mathbf{v})}\mathbf{v}_{(1)}\otimes \mathbf{v}_{(2)}$.

Using the above notation \eqref{eq:convolutionproduct},
we may write Equation~\eqref{eq:coder} as
\begin{equation}\label{eq:CDW}
Q=\sum_{k=0}^{\infty}\left ( \bar{q}_k \star \id_{S_{\cR}(V)}\right ),
\end{equation}
where the map $\bar{q}_{k}:S_{\cR}(V)\to S_{\cR}(V)$
is defined by the following commutative diagram:
\begin{equation}\label{diagram:bar}
\begin{tikzcd}
S_{\cR}(V)\arrow[d, two heads, "\pr_{k}"] \arrow[r, "\bar{q}_{k}"] & S_{\cR}(V)\\
S^{k}_{\cR}(V) \arrow[r, "q_{k}"] &S^{1}_{\cR}(V) \arrow[u, hook].
\end{tikzcd}
\end{equation}
Here $\pr_{k}:S_{\cR}(V)\to S^{k}_{\cR}(V)$ denotes
the canonical projection. We write $\id$ for $\id_{S_{\cR}(V)}$
below if there is no confusion.
We are now ready to give a detailed proof of Theorem~\ref{RSPtheorem}.

\begin{proof}[Proof of Theorem~\ref{RSPtheorem}]
For (i), by construction, it is clear that the operator
$D^\nabla$ in~\eqref{eq:DD} is indeed a degree $+1$ derivation.
Since $Q$ is a homological vector field,
from~\eqref{eq:delta}, it follows that $(\delta^\nabla)^2=0$.	
Therefore $(D^\nabla)^2=0$.

To prove (ii),
consider the case when
$\cR=\smooth{\cM}$ and $V=\sections{\tangent{\cM}}$ in Lemma~\ref{coderivation}.
Recall that $\pbwcommutator$ in~\eqref{eq:C} is $\cR$-linear,
and $\pbw^{\nabla}:\sections{S(\tangent{\cM})}\to\cD(\cM)$ is an
isomorphism of graded coalgebras over $\cR$. Since $\SdifferentialM
\in \coDer_\cR (\sections{S(\tangent{\cM})})$ and $\UdifferentialM \in
\coDer_\cR(\cD(\cM))$, it thus follows that
\[(\pbw^{\nabla})^{-1}\circ \pbwcommutator
=(\pbw^{\nabla})^{-1}\circ \UdifferentialM \circ \pbw^{\nabla}-
\SdifferentialM \in \coDer_{\cR}(\sections{S(\tangent{\cM})}).\]
Since both $ \UdifferentialM $ and $\SdifferentialM$ are of
degree $+1$ and $\pbw^{\nabla}$ is of degree $0$, it follows from
Lemma~\ref{coderivation} and Equation~\eqref{eq:CDW}
that there exists a sequence of degree $+1$ sections
$R_k\in \sections{S^k(T^\vee_\cM) \otimes T_\cM}$, $k\geqslant 0$,
such that
\begin{equation}\label{eq:ORLY}
(\pbw^{\nabla})^{-1}\circ \UdifferentialM \circ \pbw^{\nabla}-\SdifferentialM =
\sum_{k=0}^\infty (\bar{R}_k \star \id) .
\end{equation}
Here we think of
$R_{k}$ as an $\cR$-linear map
$R_{k}:\sections{S^{k}(\tangent{\cM})}\to\sections{\tangent{\cM}}$
and $\bar{R}_{k}:\sections{S(\tangent{\cM})}\to\sections{S(\tangent{\cM})}$
defined as in Diagram~\eqref{diagram:bar}.

From Equations~\eqref{formulaA1},
\eqref{formulaA2} and~\eqref{formulaA4}, it follows
that
\begin{equation}\label{firstTerms}
R_0=0, \quad R_1=0, \quad \text{and}\quad R_2 =-\atiyahcocycleQ.
\end{equation}
Thus the conclusion follows immediately from~\eqref{eq:ORLY}
by taking its $\cR$-dual.

Finally, assume that $\nabla'$ is another torsion-free affine connection.
Let $\phi:=(\pbw^{\nabla'})^{-1}\circ \pbw^{\nabla}$. Then from
Proposition~\ref{pro:2.2}, Proposition~\ref{pro:2.3} and
Theorem~\ref{thm:pbw}, it follows that
\begin{equation}\label{eq:SHA}
\phi:\big(\sections{S(\tangent{\cM})},\delta^\nabla\big)
\xto{\cong}\big(\sections{S(\tangent{\cM})},\delta^{\nabla'}\big)
\end{equation}
is an isomorphism of dg coalgebras over $(C^\infty (\cM), Q)$.
By dualizing it over the dg algebra $(C^\infty (\cM), Q)$,
we have that
\[ \phi^T: \big( \sections{\widehat{S}(T^\vee_\cM)}, D^{\nabla'} \big)
\xto{\cong} \big( \sections{\widehat{S}(T^\vee_\cM)}, D^\nabla \big) \]
is an isomorphism of dg algebras over $(C^\infty (\cM), Q)$.
This concludes the proof of the theorem.
\end{proof}

Indeed, following Kapranov \cite{MR1671737}, one may consider
$\big( \sections{\widehat{S}(T^\vee_\cM)}, D^\nabla \big)$
as the `dg algebra of functions' on the `formal neighborhood'
of the diagonal $\Delta$ of the product dg manifold
$\big(\cM\times\cM,(Q,Q)\big)$:
the PBW map $\pbw^\nabla$ is, by construction,
a formal exponential map identifying a `formal neighborhood' of the zero section of
$T_\cM$ to a `formal neighborhood' of the diagonal
of $\cM\times\cM$ as $\ZZ$-graded manifolds
and Equation~\eqref{eq:delta} asserts that $D^\nabla$ is the homological vector field
obtained on $T_\cM$ by pullback of the vector field $(Q,Q)$ on $\cM\times\cM$
through this formal exponential map.
The readers are invited to compare Theorem~\ref{RSPtheorem}
with \cite[Theorem~2.8.2]{MR1671737}.

As an immediate consequence, we are ready to prove the main result of
this section.

\begin{theorem}\label{cor:main}
Let $(\cM,Q)$ be a dg manifold.
Each choice of an affine connection $\nabla$ on $\cM$
determines an $L_\infty[1]$ algebra structure
on the space of vector fields $\XX(\cM)$.
While the unary bracket
$\lambda_1:S^{1}\big(\XX(\cM)\big)\to\XX(\cM)$
is the Lie derivative $\liederivative{Q}$ along the homological vector field,
the higher multibrackets
$\lambda_{k}:S^{k}\big(\XX(\cM)\big)\to\XX(\cM)$,
with $k\geq 2$, arise as the composition
\[ \lambda_{k}: S^{k}\big(\XX(\cM)\big)\to\sections{S^{k}(\tangent{\cM})}\xto{R_{k}} \XX(\cM)\]
induced by a family of sections $\{R_k\}_{k\geq 2}$ of the vector bundles
$S^k(T^\vee_\cM)\otimes T_\cM$ starting with $R_2=-\atiyahcocycleQ$.

Furthermore, the $L_{\infty}[1]$ algebra structures on $\XX(\cM)$
arising from different choices of affine 
connections are all isomorphic.
\end{theorem}

For clarity, we point out that $S^{k}\big(\XX(\cM)\big)$ 
denotes the symmetric tensor product \emph{over the field $\KK$} of $k$ copies of $\XX(\cM)$.
While $\lambda_1$ is merely a $\KK$-linear endomorphism of $\XX(\cM)$,
we note that, for all $k\geq 2$,
the multibracket $\lambda_{k}$ is $C^\infty(\cM)$-linear in each of its $k$ arguments.

\begin{proof}
The first part follows immediately from the fact that
$\big(\sections{S(T_\cM)}, \delta^\nabla\big)$ as in
\eqref{eq:YVR}
is a dg coalgebra over $(C^\infty (\cM), Q)$.

The uniqueness is a direct consequence of Theorem~\ref{RSPtheorem} as well.
Indeed, it is easier to derive it using the dg coalgebra
$\big(\sections{S(T_\cM)},\delta^\nabla\big)$ as in~\eqref{eq:YVR}.
If $\nabla'$ is another torsion-free affine connection on $\cM$,
we know that $\phi:\big(\sections{S(\tangent{\cM})},\delta^\nabla\big)
\xto{\cong}\big(\sections{S(\tangent{\cM})},\delta^{\nabla'}\big)$
as in~\eqref{eq:SHA}
is an isomorphism of dg coalgebras over the $(C^\infty(\cM),Q)$.
Thus it follows that the sequence of maps $\{\phi_k\}_{k\geq 1}$
defined by the composition
\[ \phi_k:S^{k}\left(\XX(\cM)\right) \to\sections{S^{k}(\tangent{\cM})} 	
\xto{\phi}\sections{S(\tangent{\cM})} \xto{\pr_1}\sections{T_\cM} =\XX(\cM) \]
is an isomorphism of $L_\infty[1]$ algebras.
Indeed, it is simple to see from~\eqref{eqnpbw} 
that the linear term $\phi_1$ is the identity map.
\end{proof}

Such an $L_{\infty}[1]$ algebra on $\XX(\cM)$ is called the
\bfemph{Kapranov $L_{\infty}[1]$ algebra} of the dg
manifold $(\cM, Q)$.

\subsection{Recursive formula for multibrackets}

It is clear that the Kapranov $L_{\infty}[1]$ algebra of a dg manifold
in Theorem~\ref{cor:main} is
completely determined by the Atiyah $1$-cocycle and
\[ R_{k}\in\sections{S^{k}(\tangent{\cM}^{\vee})\otimes\tangent{\cM}}
\cong\sections{\Hom(S^{k}(\tangent{\cM}),\tangent{\cM})} \]
for $k\geq 3$.

Recall that, for the $L_\infty[1]$ algebra
on the Dolbeault complex $\OO^{0,\bullet}(T\oz_X)$ associated
to the Atiyah class of the tangent bundle $T_X$ of a Kähler manifold $X$,
Kapranov showed that the multibrackets can be described explicitly
by a very simple formula: \eqref{eq:Rk1}.
For a general complex manifold, it was
proved in \cite{MR4271478} that they can be computed recursively
as well.
It is thus natural to ask if one can describe the multibrackets
in Theorem~\ref{cor:main} explicitly.

In what follows, we will
give a characterization of these multibrackets,
or equivalently all terms $R_{k}$, $k\geq 2$,
by showing that they are
completely determined by the Atiyah cocycle $\atiyahcocycleQ$,
the curvature $\curvature$,
and their higher covariant derivatives, by a recursive formula.

We need to introduce some notations first.

By $\widetilde{\coder}R_{n-1}\in\sections{S^{n}(\tangent{\cM}^{\vee})\otimes\tangent{\cM}}$,
we denote the symmetrized covariant derivative of $R_{n-1}$.
That is, for any $\boldX\in\sections{S^{n}(\tangent{\cM})}$,
\begin{align}\label{eq:symcovariantderivative}
\left(\widetilde{\coder}R_{n-1}\right)(\boldX)
&=\sum_{k=1}^{n}\sign_{k}
\big(\nabla_{X_k}R_{n-1}\big)(\boldX^{\{k\}}) \nonumber \\
&=\sum_{k=1}^{n}\sign_{k}
\left((-1)^{\degree{X_k}} \nabla_{X_k} 
\big(R_{n-1}(\boldX^{\{k\}})\big)-R_{n-1}
\big(\nabla_{X_k}\boldX^{\{k\}}\big)\right)
.\end{align}
Here $\sign_{k}=(-1)^{\degree{X_{k}}(\degree{X_{1}}+\cdots +\degree{X_{k-1}})}$ is the Koszul sign.

Let $\Bmap:\sections{\tangent{\cM}\otimes S(\tangent{\cM})}
\to\sections{S(\tangent{\cM})}$ be the map defined by
\begin{equation}\label{Bmap}
\Bmap(Y;\boldX)=(\pbw^{\nabla})^{-1}\left(Y\cdot\pbw^{\nabla}(\boldX)\right)-\nabla_{Y}\boldX,
\end{equation}
$\forall Y\in\XX(\cM)$ and $\boldX\in\sections{S^{n}(\tangent{\cM})}$.
The following can be verified directly.

\begin{lemma}\label{lem:SCE}
The map $\Bmap$ is well defined and ${\cR}$-linear.
Hence $\Bmap$ is indeed a bundle map
\[ \Bmap:\tangent{\cM}\otimes S(\tangent{\cM})\to S(\tangent{\cM}) .\]
\end{lemma}

As we will see below,
the map $ \Bmap$ is completely determined by the
curvature $\curvature$ and its higher covariant derivatives.

Let \[ \Gamma\big(\widehat{S}(T\dual_{\cM})\big)\otimes_{\cR}\Gamma\big(S(\tangent{\cM})\big)
\xto{\duality{\argument}{\argument}}\cR \]
be the duality pairing defined by
\[ \contraction{\alpha_1\odot\cdots\odot\alpha_q}
{X_1\odot\cdots\odot X_p} =
\begin{cases}
\sum_{\sigma\in S_p} \sign\
\duality{\alpha_{1}}{X_{\sigma(1)}}\cdot
\duality{\alpha_{2}}{X_{\sigma(2)}}\cdots
\duality{\alpha_{p}}{X_{\sigma(p)}}
& \text{if } p=q \\
0 & \text{if } p\neq q
\end{cases} \]
for all homogeneous elements $\alpha_1,\dots,\alpha_q\in\sections{\tangent{\cM}^{\vee}}$
and $X_1,\dots,X_p\in\sections{\tangent{\cM}}$.
The symbol $\sign=\sign(\alpha_{1},\alpha_{2},\cdots,\alpha_{p}, X_1,X_2,\cdots,X_p)$
denotes the Koszul signs arising from the reordering of the homogeneous
objects $\alpha_{1},\alpha_{2},\cdots,\alpha_{p}, X_1,X_2,\cdots,X_p$
in each term of the right hand side.

The following is an immediate consequence of the Fedosov
construction of graded manifolds
\cite[Theorem~5.6 and Proposition~5.2]{MR3910470}.
A short description on this topic can be found in Appendix~\ref{sec:A}.

\begin{lemma}\label{mainlemma}
\strut
\begin{enumerate}
\item[(i)] The bundle map
$\Bmap:\tangent{\cM}\otimes S(\tangent{\cM})\to S(\tangent{\cM})$
in Lemma~\ref{lem:SCE} is
completely determined by the curvature $\curvature$ and its higher
covariant derivatives. 
More precisely, given any $Y\in\XX(\cM)$, 
provided that $\Bmap(Y;\boldY)$ is known
for all $\boldY\in\sections{S^{\leq n-1}(\tangent{\cM})}$,
one can compute $\Bmap(Y;\boldX)$ for any $\boldX\in\sections{S^n(\tangent{\cM})}$.
\item[(ii)] Moreover, if $\curvature = 0$, then
$\Bmap(Y;\boldX)=Y\odot\boldX$,
for all $Y\in\XX(\cM)$ and $\boldX\in\sections{S(\tangent{\cM})}$.
\end{enumerate}
\end{lemma}
\begin{proof}
(i). Let
\[ \nabla^{\lightning}_{Y}\boldX=(\pbw^{\nabla})^{-1}(Y\cdot \pbw^{\nabla}(\boldX)) .\]
Then by Equation~\eqref{Bmap},
\[ \Bmap(Y; \boldX) =
\nabla^{\lightning}_Y \boldX-\nabla_Y \boldX .\]

For the rest of the proof, we follow the notation from Appendix~\ref{sec:A},
in particular, Theorem~\ref{LStheorem}.
For all $\sigma\in\Gamma\big(\widehat{S}(T\dual_{\cM})\big)$, we have
\begin{align*}
\duality{\sigma}{\nabla^{\lightning}_Y \boldX-\nabla_Y \boldX}
&=(-1)^{\degree{\sigma}\cdot \degree{Y}}\duality{\nabla_Y \sigma
-\nabla^{\lightning}_Y\sigma}{\boldX}\\
&=(-1)^{\degree{\sigma}\cdot \degree{Y}}\contraction{i_Y (d^{\nabla}-
d^{\nabla^{\lightning}})(\sigma)}{\boldX}\\
&=(-1)^{\degree{\sigma}\cdot \degree{Y}}\contraction{i_Y (\delta-
\widetilde{\Amap})(\sigma)}{\boldX}\\
&=\duality{\sigma}{Y\odot\boldX}-(-1)^{\degree{\sigma}\cdot\degree{Y}}
\duality{i_Y\widetilde{\Amap}(\sigma)}{\boldX} \\
&=\duality{\sigma}{Y\odot \boldX}
-\duality{\sigma}{(i_Y\widetilde{\Amap})^T \boldX}.
\end{align*}
Thus it follows that
\[ \Bmap (Y; \boldX)
=Y\odot \boldX-(i_Y\widetilde{\Amap})^T\boldX .\]
The conclusion thus follows from Corollary~\ref{cor:An}.

(ii) Moreover, if $\curvature = 0$, then
$\Amap=0$ by Equation~\eqref{eq:A}, and hence we obtain
\[ \Bmap (Y;\boldX)=Y\odot \boldX.\qedhere \]
\end{proof}

\begin{theorem}\label{theorem2}
\strut
\begin{enumerate}
\item[(i)] The sections $R_{n}\in\sections{S^{n}(\tangent{\cM}^{\vee})\otimes\tangent{\cM}}$,
with $n\geq 3$, are completely determined by the Atiyah $1$-cocycle $\atiyahcocycleQ$,
the curvature $\curvature$, and their higher covariant derivatives, through the recursive formula
\begin{equation}\label{Rn}
R_n=\frac{2}{n}(\bar{R}_{2}\star\id)
+\frac{1}{n}\sum_{k=2}^{n-1}
\left[\big(\overline{\widetilde{\coder}R}_k \star \id \big) +(1-k)(\bar{R}_{k}\star\id )
-\Bmap\circ(\bar{R}_{k}\otimes\id)\circ\Delta\right]
.\end{equation}
\item[(ii)]
In particular, if $\curvature = 0$, then $ R_{2}=-\atiyahcocycleQ$
and $R_{n}= \frac{1}{n}\widetilde{\coder} R_{n-1}$ for all $n\geq 3$.
\end{enumerate}
\end{theorem}

In terms of Sweedler's notation
$\Delta \boldX = \boldX_{(1)}\otimes \boldX_{(2)}$,
one can rewrite Equation~\eqref{Rn} as follows:
\begin{multline*}
R_{n}(\boldX) = \frac{1}{n}\sum_{k=2}^{n-1}\left[\left(\widetilde{\coder}
R_k(\boldX_{(1)})\odot\boldX_{(2)}\right)
+ (1-k)\left( R_{k}(\boldX_{(1)})\odot \boldX_{(2)}\right)
-\Bmap\left(R_k(\boldX_{(1)}); \boldX_{(2)}\right)\right] \\
+\frac{2}{n} \left( R_2 (\boldX_{(1)})\odot\boldX_{(2)} \right ).
\end{multline*}

Now we proceed to prove Theorem~\ref{theorem2}.
For any $\boldX\in \sections{S^{n}(\tangent{\cM})}$,
we can write
\begin{align}
\pbwcommutator(\boldX)&= \pbw^{\nabla}\circ \Big ((\pbw^{\nabla})^{-1}
\circ \UdifferentialM \circ \pbw^{\nabla}- \SdifferentialM \Big)(\boldX) && \nonumber\\
&=\pbw^{\nabla} \left(\sum_{k=0}^{n}(\bar{R}_k \star \id )(\boldX) \right)
&& \text{by Equation~\eqref{eq:ORLY}} \nonumber\\
&=\sum_{k=2}^{n}\pbw^{\nabla}\circ (\bar{R}_k \star \id)(\boldX)
&& \text{by Equations~\eqref{firstTerms}} 
\label{Cnabla}.
\end{align}

In order to simplify the notation,
we introduce a sequence of maps
$\Bmapindex{k}:\sections{S(\tangent{\cM})}
\to\sections{S(\tangent{\cM})}$,
for $k\geq 2$, defined by
\[ \Bmapindex{k}(\boldX)=\Bmap\circ(\bar{R}_k\otimes\id)\circ\Delta(\boldX),
\quad\forall\boldX\in\sections{S^{n}(\tangent{\cM})} .\]
Explicitly, in terms of Sweedler's notation
$\Delta\boldX=\boldX_{(1)}\otimes\boldX_{(2)}$, we write
\begin{align}
\Bmapindex{k}(\boldX)
&=\Bmap(R_{k}(\boldX_{(1)});\boldX_{(2)})\nonumber\\
&=(\pbw^{\nabla})^{-1}\left( R_{k}(\boldX_{(1)})\cdot \pbw^{\nabla}(\boldX_{(2)})\right)
-\nabla_{R_{k}(\boldX_{(1)})}\boldX_{(2)}. \label{defn:bmapindex}
\end{align}
From Lemma~\ref{lem:SCE}, it follows that $\Bmapindex{k}$, with
$k\geq 2$, is $\cR$-linear. That is, $\Bmapindex{k}$, with
$k\geq 2$, is indeed a bundle map
$S(\tangent{\cM})\to S(\tangent{\cM})$.

\begin{proof}[Proof of Theorem~\ref{theorem2}]

(i) First, we will prove the recursive formula \eqref{Rn}.

Pick any element $\boldX=X_1\odot\cdots\odot X_n$ in $\Gamma\big(S^n(T_\cM)\big)$.
Again, for the sake of simplicity, we use Sweedler's notation
$\Delta \boldX = \boldX_{(1)}\otimes \boldX_{(2)}$ and the Koszul sign
$\sign_{k}=(-1)^{\degree{X_{k}}(\degree{X_{1}}+\cdots +\degree{X_{k-1}})}$.
For each $l$, by
Equations~\eqref{formulaofpbw} and~\eqref{eq:convolutionproduct}, we have
\begin{align*}
&(n-l+1)\pbw^{\nabla}\circ (\bar{R}_{l}\star \id)(\boldX) \\
&=(n-l+1) \pbw^{\nabla}(R_{l}(\boldX_{(1)})\odot \boldX_{(2)})\\
&=R_{l}(\boldX_{(1)})\cdot \pbw^{\nabla}(\boldX_{(2)})
-\pbw^{\nabla}\left(\nabla_{R_{l}(\boldX_{(1)})}\boldX_{(2)}\right )\\
&\quad +\sum_{k=1}^{n} \sign_{k} (-1)^{\degree{X_{k}}}
\left( X_{k} \cdot \pbw^{\nabla}\left(R_{l}(\boldX_{(1)}^{\{k\}})
\odot\boldX_{(2)}^{\{k\}}\right) -\pbw^{\nabla}\left(\nabla_{X_{k}}
\left( R_{l}(\boldX_{(1)}^{\{k\}})\odot \boldX_{(2)}^{\{k\}}\right) \right) \right)\\
&=R_{l}(\boldX_{(1)})\cdot \pbw^{\nabla}(\boldX_{(2)})
-\pbw^{\nabla}\left(\nabla_{R_{l}(\boldX_{(1)})}\boldX_{(2)}\right) \\
&\quad +\sum_{k=1}^{n}\sign_{k} (-1)^{\degree{X_k}}
\left[ X_k \cdot \pbw^{\nabla}\circ(\bar{R}_l\star \id)(\boldX^{\{k\}})
-\pbw^{\nabla}\left( \nabla_{X_k} \left((\bar{R}_l\star \id)(\boldX^{\{k\}})\right)\right)\right]
.\end{align*}

Combining it with Equation~\eqref{defn:bmapindex}, we conclude that
\begin{multline*}
(n-l+1)\pbw^{\nabla}\circ(\bar{R}_{l}\star\id)(\boldX)
-\pbw^{\nabla}\circ\Bmapindex{l}(\boldX) \\
=\sum_{k=1}^{n}\sign_{k}(-1)^{\degree{X_k}} 
\left [X_k\cdot\pbw^{\nabla}\circ(\bar{R}_l\star\id)(\boldX^{\{k\}})
-\pbw^{\nabla}\left(\nabla_{X_k} 
\left((\bar{R}_l\star\id)(\boldX^{\{k\}})\right)\right)\right]
.\end{multline*}
Therefore,
\begin{multline}\label{6A}
(n-l+1) (\bar{R}_{l}\star \id)(\boldX)- \Bmapindex{l}(\boldX) \\
=\sum_{k=1}^{n}\sign_{k} (-1)^{\degree{X_k}}
\left [ (\pbw^{\nabla})^{-1} \left(X_k \cdot \pbw^{\nabla}
\circ(\bar{R}_l\star \id)(\boldX^{\{k\}}) \right)
- \nabla_{X_k} \left((\bar{R}_l\star \id)(\boldX^{\{k\}})\right) \right]
.\end{multline}

Also, for each $l$,
by Equation~\eqref{eq:symcovariantderivative}, we have
\begin{equation}
\begin{aligned}
& (\overline{\widetilde{\coder}R}_l \star \id) (\boldX) \nonumber \\
&=\sum_{k=1}^{n}\sign_{k} \left[ (\coder R_{l})(X_{k};\boldX^{\{k\}}_{(1)})
\odot \boldX^{\{k\}}_{(2)}\right] \nonumber\\
&=\sum_{k=1}^{n}\sign_{k} \left [ (-1)^{\degree{X_{k}}} \left(\left(
\nabla_{X_{k}} R_{l}(\boldX^{\{k\}}_{(1)}) \right) \odot \boldX^{\{k\}}_{(2)} \right) -
\left( R_{l}\left(\nabla_{X_k}\boldX^{\{k\}}_{(1)}\right)
\odot \boldX^{\{k\}}_{(2)}\right)\right] \nonumber \\
&=\sum_{k=1}^{n}\sign_{k} \left [ (-1)^{\degree{X_{k}}}
\left(\left( \nabla_{X_{k}} R_{l}(\boldX^{\{k\}}_{(1)}) \right)
\odot \boldX^{\{k\}}_{(2)} \right)+(-1)^{\degree{X_{k}}\cdot
\degree{\boldX^{\{k\}}_{(1)}}}\left( R_{l}(\boldX^{\{k\}}_{(1)})\odot
\left(\nabla_{X_{k}}\boldX^{\{k\}}_{(2)}\right)\right)\right] \nonumber\\
&\quad -\sum_{k=1}^{n}\sign_{k} \left [
\left( R_{l}\left(\nabla_{X_k}\boldX^{\{k\}}_{(1)} \right) \odot \boldX^{\{k\}}_{(2)}\right)
+(-1)^{\degree{X_{k}}\cdot \degree{\boldX^{\{k\}}_{(1)}}}
\left( R_{l}(\boldX^{\{k\}}_{(1)})\odot \left(\nabla_{X_{k}}
\boldX^{\{k\}}_{(2)}\right)\right)\right] \nonumber\\
&= \sum_{k=1}^{n}\sign_k\left[(-1)^{\degree{X_k}} \nabla_{X_k}
\left( (\bar{R}_l\star \id)(\boldX^{\{k\}})\right)
-(\bar{R}_l\star \id)\left(\nabla_{X_k}\boldX^{\{k\}}\right) \right] .
\end{aligned}
\end{equation}

According to~\eqref{firstTerms}, we have
$R_{2}=-\atiyahcocycleQ$. Hence
\begin{equation}\label{eq:HAN}
\pbw^{\nabla}\circ (\bar{R}_{2}\star \id)(\boldX) =
-\sum_{i<j} \sign_{i} \sign_{j}(-1)^{\degree{X_{i}}\cdot \degree{X_{j}}} 
\pbw^{\nabla}\left( \atiyahcocycleQ(X_i, X_j)\odot \boldX^{\{i,j\}}\right).
\end{equation}

By Equations~\eqref{formulaA3} and~\eqref{eq:HAN}, we have
\begin{align*}
&\pbwcommutator(\boldX) - \frac{2}{n} \pbw^{\nabla}\circ (\bar{R}_{2}\star \id)(\boldX)\\
&=\frac{1}{n}\sum_{k=1}^{n} \sign_{k}\left [ (-1)^{\degree{X_{k}}}X_{k}\cdot 
\pbwcommutator(\boldX^{\{k\}})-\pbwcommutator(\nabla_{X_k}\boldX^{\{k\}})\right] \nonumber \\
&= \frac{1}{n} \sum_{k=1}^{n}\sum_{l=2}^{n-1}\sign_k
\left[ (-1)^{\degree{X_k}} X_k \cdot \pbw^{\nabla}\circ (\bar{R}_l\star \id)(\boldX^{\{k\}})
-\pbw^{\nabla} \circ (\bar{R}_l\star \id)(\nabla_{X_k}\boldX^{\{k\}})
\right]\nonumber \\
&=\frac{1}{n} \sum_{k=1}^{n}\sum_{l=2}^{n-1}\sign_k \ (-1)^{\degree{X_k}} 
\left [X_k \cdot \pbw^{\nabla}\circ(\bar{R}_l\star \id)(\boldX^{\{k\}})
-\pbw^{\nabla}\left( \nabla_{X_k}
\left((\bar{R}_l\star \id)(\boldX^{\{k\}})\right) \right)\right] \\
&\quad +\frac{1}{n} \sum_{k=1}^{n}\sum_{l=2}^{n-1}
\sign_k\left[(-1)^{\degree{X_k}}\pbw^{\nabla} \left( \nabla_{X_k} \left(
(\bar{R}_l\star \id)(\boldX^{\{k\}})\right)\right)
-\pbw^{\nabla} \circ (\bar{R}_l\star \id)\left(\nabla_{X_k}\boldX^{\{k\}}\right) \right]
,\end{align*}
where the second equality is obtained by applying Equation~\eqref{Cnabla}
to $\pbwcommutator(\boldX^{\{k\}})$ and $\pbwcommutator(\nabla_{X_{k}}\boldX^{\{k\}})$.

It thus follows that
\begin{equation}\label{eq:45}
(\pbw^{\nabla})^{-1} \circ \pbwcommutator(\boldX)- \frac{2}{n}
(\bar{R}_{2}\star \id)(\boldX)=\alpha+\beta,
\end{equation}
where
\[ \alpha=\frac{1}{n} \sum_{k=1}^{n}\sum_{l=2}^{n-1}\sign_k (-1)^{\degree{X_k}}
\left[ (\pbw^{\nabla})^{-1}\left(X_k\cdot\pbw^{\nabla}
\circ(\bar{R}_l\star \id)(\boldX^{\{k\}})\right)
-\nabla_{X_k}\left((\bar{R}_l\star\id)(\boldX^{\{k\}})\right) \right] ,\]
and
\begin{multline}\label{eq:beta}
\beta=\frac{1}{n} \sum_{k=1}^{n}\sum_{l=2}^{n-1}\sign_k\left[(-1)^{\degree{X_k}}
\left( \nabla_{X_k} \left( (\bar{R}_l\star \id)(\boldX^{\{k\}})\right)\right)
-(\bar{R}_l\star \id)\left(\nabla_{X_k}\boldX^{\{k\}}\right) \right] \\
=\sum_{l=2}^{n-1}\frac{1}{n} \big(\overline{\widetilde{\coder} R}_l \star \id\big) (\boldX) .
\end{multline}

Now, according to~\eqref{6A},
\begin{align}
\alpha-\sum_{l=2}^{n-1} (\bar{R}_{l}\star \id)(\boldX) 
=&\sum_{l=2}^{n-1} \frac{1}{n} \left( (n-l+1) (\bar{R}_l\star \id)
(\boldX)- \Bmapindex{l}(\boldX) \right) -
\sum_{l=2}^{n-1} (\bar{R}_{l}\star \id)(\boldX) \nonumber\\
=& \frac{1}{n} \sum_{l=2}^{n-1} \left[ (1-l)\left((\bar{R}_{l}\star \id)(\boldX)\right)-
\Bmapindex{l}(\boldX)\right] \label{eq:50}.
\end{align}

Equation~\eqref{Cnabla} can be rewritten as
\[ R_{n}(\boldX)=(\pbw^{\nabla})^{-1}\circ\pbwcommutator(\boldX)
-\sum_{k=2}^{n-1}(\bar{R}_{k}\star\id)(\boldX) .\]
Equations~\eqref{eq:45}, \eqref{eq:beta} and~\eqref{eq:50}
then yield Equation~\eqref{Rn}.

From \eqref{firstTerms}, we know that $R_2=-\atiyahcocycleQ$.
According to Lemma~\ref{mainlemma},
the bundle map
$\Bmap$
is
completely determined by the curvature $\curvature$ and its higher
covariant derivatives.
It thus follows from the recursive formula \eqref{Rn}
that, for any $n\geq 3$,
$R_{n}$ is determined by $R_{k}$ with $k\leq n-1$,
their covariant derivatives and the curvature. Thus, by inductive argument,
$R_{n}$ is completely determined by the Atiyah $1$-cocycle,
the curvature and their higher covariant derivatives.

(ii) Assume that $\curvature=0$.
By Lemma~\ref{mainlemma}, the bundle map
$\Bmap:{\tangent{\cM}\otimes S(\tangent{\cM})}\to S(\tangent{\cM})$
is given by $\Bmap (Y;\boldX)=Y\odot \boldX$.
Thus the formula $R_{n}(\boldX)=\frac{1}{n}\widetilde{\coder}R_{n-1}(\boldX)$
can be obtained by induction argument, again using the recursive formula
\eqref{Rn}.

This concludes the proof of the theorem.
\end{proof}

\section{Examples}

This section is devoted to the study of examples of Kapranov
$L_{\infty}[1]$ algebras of some standard dg manifolds including
those corresponding to $L_{\infty}[1]$ algebras, foliations and
complex manifolds as in Examples~\ref{example-one} and~\ref{example-two}.

\subsection{dg manifolds associated to \texorpdfstring{$L_{\infty}[1]$}{L∞[1]} algebras}

Let $\frakg$ be a finite dimensional $L_{\infty}$ algebra with $d=\dim \frakg$.
Then $\frakg[1]$ is an $L_{\infty}[1]$ algebra: the (canonical) symmetric coalgebra
$\left( S(\frakg[1]),\Delta \right)$ is equipped with a coderivation
$\widetilde{Q} \in \coDer( S(\frakg[1]))$ of degree $+1$
satisfying $\widetilde{Q}\circ\widetilde{Q}=0$ and $\widetilde{Q}(1)=0$.
Indeed, $\widetilde{Q}$ is equivalent to a sequence of linear maps
$q_k:S^{k}(\frakg[1]) \to \frakg[1], \ k\geq 1$, of degree $+1$
satisfying the generalized Jacobi identities.
The map $q_k$ is called the $k$-th multibracket.

Given an $L_{\infty}[1]$ algebra $\frakg[1]$, we say a vector
space $\frakM$ is a $\frakg[1]$-module if
there exists a sequence of maps
$\rho_k: S^{k}(\frakg[1])\otimes\frakM\to\frakM$ of degree $+1$,
$\forall k\geq 0$, satisfying the standard compatibility condition \cite{MR1327129}.
If we write
\begin{equation}\label{eq:action}
\rho=\sum_{k\geq 0}\rho_{k}:S(\frakg[1])\otimes\frakM\to\frakM,
\end{equation}
the compatibility condition is expressed
explicitly as
\[ \rho\circ\left((\id_{S(\frakg[1])}\otimes\rho)\circ(\Delta\otimes\id_{\frakM})
+\widetilde{Q}\otimes\id_{\frakM}\right)=0 .\]

As an obvious example, we have the \bfemph{trivial module}
$\frakM=\KK$ together with the trivial action $\rho_k=0$ for all $k\geq 0$.
Another example is the \bfemph{adjoint module}
$\frakM=\frakg[1]$ with the adjoint action
$\rho_k: S^{k}(\frakg[1])\otimes \frakg[1]\to \frakg[1]$ defined by
\[ \rho_k(\boldX\otimes X)=q_{k+1}(\boldX\odot X) ,\]
where $\boldX\in S^{k}(\frakg[1])$, $X\in\frakg[1]$ 
and $q_{k+1}:S^{k+1}(\frakg[1])\to\frakg[1]$ is
the multibracket of the $L_{\infty}[1]$ algebra $\frakg[1]$.
That is, $\{\rho_k\}_{k\geq 0}$ is defined by
the following commutative diagram
\[ \begin{tikzcd}
S^{k}(\frakg[1])\otimes \frakg[1] \arrow[rr, "\rho_k"] \arrow[rd, "\sym"] & & \frakg[1] \\
&S^{k+1}(\frakg[1])\arrow[ru, "q_{k+1}"] &
\end{tikzcd} \]
where $\sym: S^{\bullet}(\frakg[1])\otimes \frakg[1] \to S^{\bullet+1}(\frakg[1])$ 
is the canonical symmetrization map.
By taking its dual, $(\frakg[1])^\vee$ is also a $\frakg[1]$-module,
where the action is called the \bfemph{coadjoint action}.	

Throughout this section, we denote the degree
of a homogeneous element $x\in\frakg[1]$ by $\degree{x}$.
In particular, if $\frakg$ is a Lie algebra concentrated
in degree $0$, then for any $x\in\frakg[1]$, its degree is $\degree{x}=-1$.

The associated Chevalley--Eilenberg cochain complex
of a $\frakg[1]$-module $\frakM$ is
\[ \cC(\frakg[1]; \frakM)=\Big(\Hom \big (S(\frakg[1]), \frakM \big), d_{\CE}^{\frakM}\Big), \]
where $d_{\CE}^{\frakM}$ is defined by
\[ d_{\CE}^{\frakM}(F)=\rho\circ (\id\otimes F)
\circ \Delta -(-1)^{\degree{F}}F\circ \widetilde{Q} ,\]
for any homogeneous element
$F\in \Hom \big (S(\frakg[1]), \frakM \big)$.

Observe that when $\frakM$ is the trivial module $\KK$,
the associated Chevalley--Eilenberg cochain complex
\[\cC(\frakg[1]; \KK)=\Big(\Hom \big (S(\frakg[1]), \KK \big), d_{\CE}^{\KK}=d_{\CE}\Big) \]
is a dg algebra, equipped with the multiplication
\begin{equation}\label{product}
f\odot g=\mu_{\KK}\circ(f\otimes g)\circ\Delta : S(\frakg[1])\to\KK
\end{equation}
for any $f,g\in\Hom(S(\frakg[1]),\KK)$.
In other words, the dg algebra
$(\smooth{\frakg[1]},Q)$ coincides with the
Chevalley--Eilenberg cochain complex
$\big(\cC(\frakg[1];\KK),d_{\CE}\big)$
of the trivial $\frakg[1]$-module $\KK$.
That is, $\big(\cC(\frakg[1];\KK),d_{\CE}\big)$
is the dg algebra dual to
the dg coalgebra $(S(\frakg[1]),\widetilde{Q})$.
Moreover, for any $\frakg[1]$-module $\frakM$,
the Chevalley--Eilenberg cochain complex
$\big(\cC(\frakg[1]; \frakM), d_{\CE}^{\frakM}\big)$
is a dg module over the dg algebra $(\smooth{\frakg[1]},Q)$,
where the action, under the identification
$\mu _{0}:\KK\otimes \frakM \cong \frakM$, is given by
\begin{equation}\label{action}
f\cdot F =\mu _{0}\circ (f\otimes F)\circ \Delta: S(\frakg[1])\to \frakM
\end{equation}
for any $f \in \Hom(S(\frakg[1]),\KK)$ and $F\in \Hom(S(\frakg[1]), \frakM)$.
In particular, this means that it satisfies
the compatibility condition
\begin{equation}\label{eq:Leibnizrule}
d_{\CE}^{\frakM}(f\cdot F)=d_{\CE}(f)\cdot F +(-1)^{\degree{f}}f \cdot d_{\CE}^{\frakM}(F).
\end{equation}
Therefore, the Chevalley--Eilenberg differential $d_{\CE}^{\frakM}$
is completely determined by its image of elements in $\frakM$,
which is essentially induced by the action \eqref{eq:action}.
More precisely, for any $x\in\frakM$,
\[ d_{\CE}^{\frakM}(x)=\sum_{k\geq 0}\rho_k(\argument,x)\in\Hom(S(\frakg[1]),\frakM) .\]
In particular, if $\frakM=\frakg[1]$ is the adjoint module 
of the finite dimensional $L_{\infty}[1]$ algebra $\frakg[1]$ described above, 
the Chevalley--Eilenberg differential 
$d_{\CE}^{\frakg[1]}$ (seen as an operator on $\widehat{S}(\frakg[1])^\vee\otimes\frakg[1]$)
is determined by the relation
\begin{equation}\label{eq:ORD}
d_{\CE}^{\frakg[1]}(x)=\sum_{k=1}^{\infty}
\frac{1}{(k-1)!}\xi^{i_{k-1}}\odot \cdots \odot \xi^{i_{1}}\otimes
q_{k}(e_{i_{1}}\odot \cdots \odot e_{i_{k-1}}\odot x), \quad \forall x\in
\frakg[1],
\end{equation}
where $\{e_{1},\cdots, e_{d}\}$ is a basis for $\frakg[1]$ and
$\{\xi^{1},\cdots, \xi^{d}\}$ is the dual basis for $(\frakg[1])^\vee$.
In Equation~\eqref{eq:ORD} and in the remainder of the present section, 
we use the Einstein notation tacitly to avoid inserting summations over the indices 
$i_1,\dots,i_{k-1}$ in many equations.

\begin{remark}
In terms of Sweedler's notation, we may write \eqref{product} as
\[(f\odot g)(\boldX)=\sum_{(\boldX)} (-1)^{\degree{g} 
\cdot \degree{\boldX_{(1)}}} f(\boldX_{(1)}) g(\boldX_{(2)})\]
and \eqref{action} as
\[(f\cdot F)(\boldX)=\sum_{(\boldX)}(-1)^{\degree{F}
\cdot \degree{\boldX_{(1)}}} f(\boldX_{(1)})F(\boldX_{(2)}), \]
where $f,g\in\Hom(S(\frakg[1]),\KK)$, $F\in\Hom(S(\frakg[1]),\frakM)$,
$\boldX\in S(\frakg[1])$ are homogeneous elements
and $\Delta\boldX=\sum\limits_{(\boldX)}\boldX_{(1)}\otimes\boldX_{(2)}$.
\end{remark}

We now proceed to describe the Kapranov $L_\infty[1]$ algebra
of the dg manifold $(\frakg[1],d_{\CE})$.
Recall that $Q=d_{\CE}$ is defined by
\begin{equation}\label{eq:Q}
Q(f)=d_{\CE}(f) = -(-1)^{\degree{f}}f\circ \widetilde{Q}
\end{equation}
for any homogeneous element $f\in \Hom(S(\frakg[1]), \KK)
\cong C^{\infty}(\frakg[1])$.

Let $\{e_{1},\cdots, e_{d}\}$ be a basis
of $\frakg[1]$ and $\{x^{1},\cdots, x^{d}\}$ its
induced coordinate functions on $\frakg[1]$ satisfying
\[ x^{i}(e_{j})= \duality{ x^{i}}{e_{j}}=\left\{
\begin{array}{ll}
1 & \text{if }\, i=j\\
0 & \text{if } \, i\neq j
\end{array}
\right. .\]
We also use the notation
\begin{equation}\label{eq:ij}
\frac{\partial}{\partial x^{j}}x^{i}:=(-1)^{\degree{x^{i}}\cdot
\degree{x^{j}}} \duality{x^{i}}{e_{j}}
.\end{equation}

\begin{lemma}\label{lem:Q}
Under the above notation, write the multibrackets as
\[ q_{k}(e_{i_{1}},\cdots, e_{i_{k}})=\sum_j c_{i_{1}\cdots i_{k}}^{j}e_{j},
\quad \forall k\geq 1 .\]
Then the homological vector field $Q\in \XX(\frakg[1])$ can be
written as
\[ Q=-\sum_j \sum_{k=1}^{\infty} \frac{1}{k!} c_{i_{1}\cdots i_{k}}^{j}x^{i_{k}}
\odot \cdots \odot x^{i_{1}}\frac{\partial}{\partial x^{j}} .\]
Here, we are making tacit use of the Einstein summation convention for the indices $i_1,\dots,i_k$.
\end{lemma}

\begin{proof}
As a vector field, $Q$ can be written as
$Q=\sum_j Q^{j}\frac{\partial}{\partial x^{j}}$
for some $Q^{j}\in C^{\infty}(\frakg[1])$. Then,
as a derivation of $C^{\infty}(\frakg[1])$, $Q$ satisfies
$Q(x^{j})=(-1)^{\degree{x^{j}}} Q^{j}$ according to~\eqref{eq:ij}.
On the other hand, according to \eqref{eq:Q}, we have
\begin{align*}
\duality{Q(x^{j})}{e_{i_{1}}\odot \cdots \odot e_{i_{k}}} 
& = -(-1)^{\degree{x^{j}}} \duality{x^{j}}{\widetilde{Q}(e_{i_{1}}\odot \cdots \odot e_{i_{k}})}\\
&=-(-1)^{\degree{x^{j}}} c_{i_{1}\cdots i_{k}}^{j}
\end{align*}
for any $k\geq 1$.

Therefore, we may conclude that
\[ Q^{j}=-\sum_{k=1}^{\infty} \frac{1}{k!} c_{i_{1}\cdots 
i_{k}}^{j} x^{i_{k}}\odot \cdots \odot x^{i_{1}} .\]
This completes the proof.
\end{proof}

Note that we have a canonical trivialization of the tangent bundle
\begin{equation}\label{eq:NYC}
\tangent{\frakg[1]}\cong \frakg[1]\times \frakg[1].
\end{equation}
Hence, we have the following identification
\begin{equation}\label{map:vectorfields}
\begin{aligned}
\smooth{\frakg[1]}\otimes\frakg[1]\from
&\XX(\frakg[1])\to\Hom\left(S(\frakg[1]),\frakg[1]\right) \\
f\otimes e_{i}\mapsfrom & \ \ f \frac{\partial}{\partial x^{i}}\ 
\mapsto \left(\boldX\mapsto (-1)^{\degree{e_{i}}\cdot 
\degree{\boldX}}\duality{f}{\boldX}\cdot e_{i} \right),
\end{aligned}
\end{equation}
where $f\in\Hom(S(\frakg[1]),\KK)\cong C^{\infty}(\frakg[1])$
is homogeneous and $\boldX\in S(\frakg[1])$.

\begin{lemma}\label{lem:Liederivative}
Under the identification \eqref{map:vectorfields},
the Lie derivative $\liederivative{Q}=[Q,\argument] \in\End(\XX(\frakg[1]))$
corresponds to the Chevalley--Eilenberg differential $d_{\CE}^{\frakg[1]}$,
where $\frakg[1]$ acts on
$\frakg[1]$ by the adjoint action.
\end{lemma}
\begin{proof}
Recall that the Chevalley--Eilenberg differential $d_{\CE}^{\frakg[1]}$ on
$\frakg[1]$ satisfies \eqref{eq:Leibnizrule}.
On the other hand, we have
\[ \liederivative{Q}(f\cdot F)=\left[ Q, f\cdot F \right] 
=Q(f)\cdot F +(-1)^{\degree{f}}f\cdot [Q,F] 
=Q(f)\cdot F +(-1)^{\degree{f}}f\cdot \liederivative{Q}(F), \]
for any homogeneous element $f\in\smooth{\frakg[1]}\cong\Hom(S(\frakg[1]),\KK)$
and $F\in\XX(\frakg[1])\cong\Hom(S(\frakg[1]),\frakg[1])$.
Since $Q(f)=d_{\CE} (f)$ according to Equation~\eqref{eq:Q},
it suffices to prove the claim for each
$\frac{\partial}{\partial x^{i}}$, $i=1,\ldots,d$.

We keep the notation $Q=\sum_j Q^{j}\frac{\partial}{\partial x^{j}}$.
Now, by Lemma~\ref{lem:Q}, we have
\begin{align*}
\liederivative{Q} \left( \frac{\partial}{\partial x^{i}} \right) &=
-(-1)^{\degree{\frac{\partial}{\partial x^{i}}}}\sum_j
\frac{\partial}{\partial x^{i}}(Q^{j})\frac{\partial}{\partial x^{j}}\\
&= -(-1)^{\degree{\frac{\partial}{\partial x^{i}}}}
\left( -\sum_j\sum_{k=1}^{\infty}\frac{1}{k!} c_{i_{1}\cdots i_{k}}^{j}
\frac{\partial(x^{i_{k}}\odot \cdots\odot x^{i_{1}})}{\partial x^{i}}
\frac{\partial}{\partial x^{j}}\right) \\
&=(-1)^{\degree{ \frac{\partial}{\partial x^{i}}}+\degree{x^{i}}}
\sum_j\sum_{k=1}^{\infty}\frac{1}{(k-1)!}
c_{i_{1}\cdots i_{k-1} i}^{j} x^{i_{k-1}}\odot \cdots \odot x^{i_{1}}
\frac{\partial}{\partial x^{j}}\\
&=\sum_j\sum_{k=1}^{\infty}\frac{1}{(k-1)!}
c_{i_{1}\cdots i_{k-1} i}^{j} x^{i_{k-1}}\odot \cdots \odot x^{i_{1}}
\frac{\partial}{\partial x^{j}}
.\end{align*}
The conclusion thus follows immediately by comparing the equation above
with~\eqref{eq:ORD}.
\end{proof}

The trivialization of the tangent bundle
\eqref{eq:NYC} induces an isomorphism
\[ \tangent{\frakg[1]}^\vee\otimes\End(\tangent{\frakg[1]})
\xto{\cong} \frakg[1] \times \big((\frakg[1])^\vee \otimes
(\frakg[1])^\vee \otimes \frakg[1] \big) \]
of vector bundles. Lemma~\ref{lem:Liederivative}, comparing with~\eqref{eq:cQ},
indicates that we have an isomorphism of cochain complexes:
\[ \big( \sections{ \frakg[1]; T^\vee_{\frakg[1]}\otimes
\End(\tangent{\frakg[1]})}^{\bullet}, \cQ \big) 
\xto{\cong} \big(\Hom^\bullet (S(\frakg[1]), \frakM ), d_{\CE}^{\frakM}\big) ,\]
where $\frakM=(\frakg[1])^{\vee}\otimes (\frakg[1])^{\vee}\otimes \frakg[1]$
is the tensor product of adjoint and coadjoint modules.

Thus we have the following

\begin{corollary}\label{cor:Qcohomology}
Let $(\cM,Q)=(\frakg[1],d_{\CE})$ be the dg manifold
corresponding to a finite-dimensional $L_\infty[1]$ algebra
$\frakg[1]$.
There is a canonical isomorphism, for any $k\in \ZZ$,
\[ \cohomology{k}\big(\Gamma(T^\vee_{\frakg[1]} \otimes\End (T_{\frakg[1]}))^\bullet,\cQ\big) \cong
\cohomology{k}_{\CE}(\frakg[1], (\frakg[1])^\vee\otimes (\frakg[1])^\vee\otimes
\frakg[1]) \]
where the right hand side stands for the Chevalley--Eilenberg
cohomology of the $L_\infty[1]$ algebra $\frakg[1]$
with values in $(\frakg[1])^\vee\otimes (\frakg[1])^\vee\otimes
\frakg[1]$.
\end{corollary}

\begin{remark}
It is sometimes useful to use the Chevalley--Eilenberg cohomology
of $L_\infty$ algebra rather than $L_\infty[1]$ algebra.
Then Corollary~\ref{cor:Qcohomology} can be rephrased as follows.

For any finite-dimensional $L_\infty$ algebra
$\frakg$, there is a canonical isomorphism, for any $k\in\ZZ$,
\[ \cohomology{k}\big(\Gamma(T^\vee_{\frakg[1]}\otimes\End(T_{\frakg[1]}))^\bullet,\cQ\big)
\cong\cohomology{k-1}_{\CE}(\frakg,\frakg^\vee\otimes\frakg^\vee\otimes\frakg), \]
where the right hand side stands for the Chevalley--Eilenberg
cohomology of the $L_\infty$ algebra $\frakg$
with values in $\frakg^\vee\otimes \frakg^\vee\otimes \frakg$.
Note that there is a degree shifting here.
\end{remark}

We still keep the notation $d_{\CE}=Q=\sum_l Q^{l}\frac{\partial}{\partial x^{l}}$.
Let $\nabla:\XX(\frakg[1])\otimes\XX(\frakg[1])\to\XX(\frakg[1])$ be the trivial
(torsion-free) connection:
$\nabla_{\frac{\partial}{\partial x^{i}}}\frac{\partial}{\partial x^{j}}=0$.
The corresponding Atiyah $1$-cocycle 
$\atiyahcocyclegQ\in\sections{\Hom\left(S^{2}(\tangent{\frakg[1]}),\tangent{\frakg[1]}\right)}$
is completely determined by the relations
\begin{align}
\atiyahcocyclegQ\left(\frac{\partial}{\partial x^{i}},\frac{\partial}{\partial x^{j}}\right)
&=-(-1)^{\degree {x^{i}}}\nabla_{\frac{\partial}{\partial x^{i}}}\liederivative{Q}
\left(\frac{\partial}{\partial x^{j}}\right) \nonumber \\
&=\sum_l (-1)^{\degree{x^{i}}+\degree{x^{j}}}
\frac{\partial^{2}}{\partial x^{i}\partial x^{j}}(Q^{l})
\frac{\partial}{\partial x^{l}}\label{eq:atiyahg1}\\
&=\sum_l (-1)^{\degree{x^{i}}+\degree{x^{j}}} 
\frac{\partial^{2}}{\partial x^{i}\partial x^{j}}
\left( -\sum_{k=1}^{\infty}\frac{1}{k!}c^{l}_{i_{1}\cdots i_{k}}x^{i_{k}}
\odot \cdots\odot x^{i_{1}}\right)\frac{\partial}{\partial x^{l}}\nonumber \\
&=-\sum_l\sum_{k=2}^{\infty}\frac{1}{(k-2)!}c^{l}_{i_{1}\cdots i_{k-2}ij}x^{i_{k-2}}
\odot \cdots \odot x^{i_{1}}\frac{\partial}{\partial x^{l}} ,\label{eq:IAD}
\end{align}
for all $i,j\in\{1,\ldots,d\}$.

Let $\widehat{\atiyahcocyclegQ}$ be the map defined
by the following commutative diagram
\[ \begin{tikzcd}
\smooth{\frakg[1]}\otimes S^{2}(\frakg[1]) \arrow[r, "\simeq"] 
& \sections{S^{2}(\tangent{\frakg[1]})} \arrow[r, "\atiyahcocyclegQ"] 
& \XX(\frakg[1]) \arrow[d, "\simeq"] \\
S^{2}(\frakg[1]) \arrow[u, hook] \arrow[rr, "\widehat{\atiyahcocyclegQ}"]
&& \Hom(S(\frakg[1]), \frakg[1])
.\end{tikzcd} \]

Equation~\eqref{eq:IAD} implies that
\[ \widehat{\atiyahcocyclegQ} (e_i,e_j): \quad
e_{l_{1}}\odot \cdots \odot e_{l_{k}}
\mapsto -q_{k+2}(e_{i}\odot e_{j}\odot e_{l_{1}}\odot \cdots \odot e_{l_{k}}) .\]

Therefore, under the identification above, we have
\[\widehat{\atiyahcocyclegQ}(x,y): \quad \boldX\mapsto
-q_{n+2}(x\odot y \odot \boldX), \]
for any $x,y\in \frakg[1]$ and $\boldX\in S^{n}(\frakg[1])$.
Thus, by abuse of notation, we may write
\[ \atiyahcocyclegQ = -\sum_{k\geq 2} q_{k} .\]

\begin{proposition}\label{pro:Atiyah}
Let $\frakg[1]$ be an $L_{\infty}[1]$ algebra with multibrackets
$q_k:S^{k}(\frakg[1]) \to \frakg[1]$, $k\geq 1$.
Then the Atiyah class $\atiyahclass{(\frakg[1], d_{\CE})}$
of the dg manifold $(\frakg[1], d_{\CE})$
is
\[ \atiyahclass{(\frakg[1], d_{\CE})} =-\big[\sum_{k\geq 2} q_{k}\big]
\in \cohomology{1}_{\CE}\big(\frakg[1], (\frakg[1])^\vee\otimes (\frakg[1])^\vee\otimes
\frakg[1]\big) \cong
\cohomology{1}\big(\Gamma(T^\vee_{\frakg[1]}\otimes\End T_{\frakg[1]})^\bullet,\cQ\big). \]
\end{proposition}

\begin{remark}
We can rephrase Proposition~\ref{pro:Atiyah} in terms of
multibrackets of $L_{\infty}$ algebra $\frakg$ instead of
$L_{\infty}[1]$ algebra $\frakg[1]$.
For a finite dimensional $L_{\infty}$ algebra $\frakg$
equipped with multibrackets $l_{k}:\Lambda^{k}\frakg\to\frakg$
of degree $2-k$ for $k\geq 1$,
the Atiyah class $\atiyahclass{(\frakg[1],d_{\CE})}$
of the dg manifold $(\frakg[1],d_{\CE})$ is
\[ \atiyahclass{(\frakg[1], d_{\CE})} =\big[\sum_{k\geq 2} l_{k}\big]
\in \cohomology{0}_{\CE}\big(\frakg, \frakg^\vee\otimes \frakg^\vee\otimes
\frakg\big) \cong
\cohomology{1}\big(\Gamma(T^\vee_{\frakg[1]}\otimes\End T_{\frakg[1]})^\bullet,\cQ\big) ,\]
where $\cohomology{0}_{\CE}\big(\frakg, \frakg^\vee\otimes \frakg^\vee\otimes
\frakg\big)$ denotes the $0$-th Chevalley--Eilenberg cohomology of
the $L_{\infty}$ algebra $\frakg$ with values in
the tensor product of adjoint and coadjoint modules 
$\frakg^\vee\otimes \frakg^\vee\otimes\frakg$.
\end{remark}	

Since the trivial connection $\nabla$ is flat, by the second part
of Theorem~\ref{theorem2}, we know that
\[R_{n}= \frac{1}{n}\widetilde{\coder} R_{n-1}
\in\sections{\Hom(S^{n}(\tangent{\frakg[1]}), \tangent{\frakg[1]})}\]
for $n\geq 3$.
As the connection $\nabla$ is trivial,
Equation~\eqref{eq:symcovariantderivative} implies that
\begin{align*}
\widetilde{\coder}R_{n-1} \left( \frac{\partial}{\partial x^{i_{1}}}\odot
\cdots\odot\frac{\partial}{\partial x^{i_{n}}}\right)&=
\sum_{k=1}^{n}\sign_{k}(-1)^{\degree {x^{i_{k}}}}\nabla_{\frac{\partial}{\partial x^{i_{k}}}}
\left(R_{n-1}\left(\frac{\partial}{\partial x^{i_{1}}}\odot
\cdots \widehat{\frac{\partial}{\partial x^{i_{k}}}} \cdots
\odot \frac{\partial}{\partial x^{i_{n}}}\right) \right) \\
&= \sum_{k=1}^{n}\sign_{k}(-1)^{\degree{x^{i_{k}}}}
\frac{\partial}{\partial x^{i_{k}}}\left(R_{n-1} 
\left(\frac{\partial}{\partial x^{i_{1}}}\odot \cdots 
\widehat{\frac{\partial}{\partial x^{i_{k}}}} \cdots \odot 
\frac{\partial}{\partial x^{i_{n}}}\right) \right)
\end{align*}

Here, $\epsilon_{k}=(-1)^{\degree{x^{i_{k}}}\left(\degree{x^{i_{1}}}\cdots
+\degree{x^{i_{k-1}}}\right)}$ is the Koszul sign.
Starting from
\[ R_{2}\left(\frac{\partial}{\partial x^{i_{1}}}\odot
\frac{\partial}{\partial x^{i_{2}}}\right)
= -(-1)^{\degree{x^{i_{1}}}+\degree{x^{i_{2}}}}
\sum_j\frac{\partial^{2}Q^{j}}{\partial x^{i_{1}}
\partial x^{i_{2}}}\frac{\partial}{\partial x^{j}} ,\]
as in~\eqref{eq:atiyahg1}, we inductively obtain that
\[R_{n}\left(\frac{\partial}{\partial x^{i_{1}}}\odot\cdots\odot 
\frac{\partial}{\partial x^{i_{n}}}\right)
=-(-1)^{\degree{x^{i_{1}}}+\cdots + \degree{x^{i_{n}}}} \sum_j
\frac{\partial ^{n}Q^{j}}{\partial x^{i_{1}}\cdots
\partial x^{i_{n}}}\frac{\partial}{\partial x^{j}}.\]

According to Corollary~\ref{cor:main}, we obtain the following

\begin{proposition}\label{thm:Linfty}
Let $\frakg[1]$ be a finite dimensional
$L_{\infty}[1]$ algebra with multibrackets
$q_k:S^{k}(\frakg[1]) \to \frakg[1]$, $k\geq 1$.
Let $(\cM,Q)=(\frakg[1], d_{\CE})$ be its corresponding
dg manifold. Choose the trivial connection.
Then the multibrackets $\{\lambda_n\}_{n\geq 1}$
of the Kapranov $L_{\infty}[1]$ algebra
structure on $\Hom \left(S(\frakg[1]),\frakg[1] \right) \cong
\widehat{S}(\frakg[1])^\vee \otimes \frakg[1]$, being identified with
$\XX(\frakg[1])$ as in Equation~\eqref{map:vectorfields}, are given as follows.
\begin{enumerate}
\item The unary bracket $\lambda_1$ coincides with the
Chevalley--Eilenberg differential with values in the $L_\infty[1]$-adjoint module
$\frakg[1]$:
\[ \lambda_1=d_{\CE}^{\frakg[1]}: \widehat{S}(\frakg[1])^\vee \otimes \frakg[1]\to
\widehat{S}(\frakg[1])^\vee \otimes \frakg[1] \]
\item For any $n\geq 2$, $\lambda_n$ is
$\widehat{S}(\frakg[1])^\vee$-linear in each of its $n$ argument,
and therefore can be considered as a linear map
\[ \lambda_n: S^n (\frakg[1])
\to \widehat{S}(\frakg[1])^\vee \otimes \frakg[1] \]
which is completely determined by
\[ \lambda_{n}(\boldX)=\sum_{k=n}^{\infty}q_k(\boldX\odot\argument),
\quad n\geq 2 ,\]
where $\boldX\in S^{n}(\frakg[1])$, and each
$q_k(\boldX\odot\argument):S^{k-n}(\frakg[1])\to\frakg[1]$
is defined by
$\boldY\mapsto q_k(\boldX\odot\boldY)$
for all $\boldY\in S^{k-n}(\frakg[1])$.
\end{enumerate}
\end{proposition}

\begin{example}
If $\frakg$ is a finite dimensional Lie algebra, then
the Kapranov $L_{\infty}$ algebra (i.e.\ $(-1)$-shifted Kapranov $L_{\infty}[1]$
algebra) of the dg manifold $(\frakg[1],d_{\CE})$
is the dgla $\Lambda \frakg^\vee \otimes \frakg$,
where the differential is the Chevalley--Eilenberg differential $d^{\frakg}_{\CE}$
of the $\frakg$-module $\frakg$ (for the adjoint action),
and the Lie bracket is $[\xi\otimes x,\eta\otimes y]=
\xi\wedge\eta\otimes [x,y]$ for homogeneous 
$\xi,\eta\in\Lambda\frakg^{\vee}$ and $x,y\in\frakg$.
\end{example}

\subsection{dg manifolds associated to complex manifolds
and integrable distributions}

Every complex manifold $X$ determines a dg 
manifold $(T_{X}^{0,1}[1],\bar{\partial})$---see Example~\ref{example-two}.
This section is devoted to the description 
of the corresponding Kapranov
$L_\infty[1]$ algebra.
Recall that for a Kähler manifold $X$, Kapranov obtained an
explicit description of an $L_\infty[1]$ algebra structure
on the Dolbeault complex $\OO^{0,\bullet}(T\oz_X)$,
where the unary bracket is the Dolbeault
operator $\bar{\partial}$
and the binary bracket is the Dolbeault cocycle of
the Atiyah class of $T_X$ \cite[Theorem~2.6]{MR1671737}.
Kapranov proved the existence of an $L_\infty[1]$ algebra
structure associated to the Atiyah class of the
holomorphic tangent bundle of any complex manifold
using formal geometry and PROP \cite[Theorem~4.3]{MR1671737}.
See Theorem~\ref{thm:complex}
below for the Dolbeault representations.
Since $T_X^{0,1} \subset T_\CC X$ is a complex integrable
distribution, we will consider general integrable distributions over $\KK$.
Indeed such $L_\infty[1]$ algebra
structures can be obtained in a more general perspective in terms
of Lie pairs \cite{MR4271478}. We recall its construction briefly below.

Let $F \subseteq {\TkM}$ be an integrable distribution.
Then $(F[1], d_F)$ is a dg manifold,
whose algebra of smooth functions $C^\infty({F[1]},\KK)$ is
identified with $\OmegaF:= \Gamma(\wedge F^\vee)$ and
the homological vector field is the
leafwise de Rham differential, i.e.\ the Chevalley--Eilenberg
differential $d_F \colon \Omega_F^\bullet \to \Omega_F^{\bullet+1}$ of
the Lie algebroid $F$.
It is well known that the normal bundle $B:= {\TkM}/F$ is naturally
an $F$-module, where the $F$-action is known as the \bfemph{Bott
connection}~\cite{MR3439229}, defined by
\[ \nabla^{\operatorname{Bott}}_a b=\prB([a,\tilde{b}]) ,\] 
for all $a \in \Gamma(F)$, $b \in \Gamma(B)$ and
$\tilde{b} \in \Gamma({\TkM})$ such that $\prB(\tilde{b}) = b$.
Here $\prB: \TkM\to B$ denotes the canonical projection.	
Let $\cD(M)$ be the space of $\KK$-linear differential operators on $M$,
and $R=C^\infty(M;\KK)$ be the space of $\KK$-valued smooth functions on $M$.
Then $\cD(M)$ is an $R$-coalgebra equipped with the standard coproduct
\begin{equation}\label{Eq:COPRODUCTonDM}
\Delta\colon\cD(M)\to\cD(M)\otimes_{R}\cD(M).
\end{equation}
Let $\cD(M)\Gamma(F)\subseteq\cD(M)$ be the left ideal of $\cD(M)$
generated by $\Gamma(F)$. Since
\[ \Delta(\cD(M)\Gamma(F))\subseteq\cD(M)\otimes_R\cD(M)\Gamma(F)
+\cD(M)\Gamma(F) \otimes_R\cD(M) ,\]
the coproduct \eqref{Eq:COPRODUCTonDM}
descends to a well-defined coproduct over $R$
\begin{equation}\label{Eq:COPRODUCTonDB}
\Delta:\cD(B)\to\cD(B)\otimes_{R}\cD(B),
\end{equation}
on the quotient space
$\cD(B):=\frac{\cD(M)}{\cD(M)\Gamma(F)}$.
Hence $\cD(B)$ is an $R$-coalgebra as well, called the $R$-coalgebra
of \textit{differential operators transverse to $F$}~\cite{MR3313214}.

It is well known that $\cD(B)$ is an $F$-module \cite{MR4271478,MR2989383},
where the $F$-action is given by
\begin{equation}\label{eq:Faction}
a \cdot \overline{u} = \overline{a \circ u},
\end{equation}
for any $a\in\Gamma(F)$ and $u\in\cD(M)$ --- 
the symbol $\overline{x}$ denotes the image of $x$ under the quotient map $\cD(M)\to\cD(B)$.
Here $\circ$ denotes the composition of differential operators.
Moreover, $F$ acts on $\cD(B)$ by coderivations.
Indeed, the associated Chevalley--Eilenberg differential
\[ d_F^\cU: \quad \Omega_F^\bullet(\cD(B))\to\Omega_F^{\bullet+1}(\cD(B)) \]
is a coderivation of the $\Omega_F$-linear coproduct
\[ \Delta: \Omega_F(\cD(B))
\to\Omega_F(\cD(B))\otimes_{\Omega_F} \Omega_F(\cD(B)) \]
extending the coproduct~\eqref{Eq:COPRODUCTonDB} on $\cD(B)$.
Thus $(\Omega_F(\cD(B)),d_F^\cU,\Delta)$ is a dg coalgebra over
$(\Omega_F,d_F)$.

Let $j\colon B\to{\TkM}$ be a splitting
of the short exact sequence of vector bundles over $M$:
\begin{equation}\label{SES}
0 \to F \xrightarrow{i} {\TkM} \xrightarrow{\prB} B \to 0.
\end{equation}
Choose a torsion-free linear connection $\nabla^B$ of the vector
bundle $B$, i.e.\ a ${\TkM}$-connection on $B$ satisfying the condition:
\[ \nabla^B_X \big(\prB(Y)\big)- \nabla^B_Y \big(\prB(X)\big) -\prB\big([X,Y]\big)=0 ,\]
for any $X, Y\in \sections{\TkM}$. It is known \cite[Lemma~5.2]{MR4271478}
that a torsion-free linear connection $\nabla^B$
automatically extends the Bott representation of $F$ on $B$,
that is, $\nabla^B_a \bar{X}=\nabla^{\Bott}_a \bar{X}$,
$\forall a\in \sections{F}$ and $\bar{X}\in \sections{B}$.

According to~\cite{MR4271478,MR2989383},
the pair $(j,\nabla^B)$ determines an
isomorphism of $R$-coalgebras
\[
\overline{\pbw}\colon \Gamma(S(B)) \to \cD(B),
\]
called the PBW isomorphism for the Lie pair $({\TkM}, F)$,
which is defined recursively by the relations
\begin{align*}
&\overline{\pbw}(f) = f, \quad \forall f \in R, \\
&\overline{\pbw}(b) = \overline{j(b)} , \quad \forall b \in \Gamma(B), \\
\intertext{and}
&\overline{\pbw}(b_1 \odot \cdots \odot b_n) = \frac{1}{n}\sum_{k=1}^n \left\{ j(b_k)
\cdot \overline{\pbw}(b^{\{k\}}) - \overline{\pbw}(\nabla^B_{ j(b_k) }(b^{\{k\}})) \right\},
\end{align*}
where we keep the notation from~\eqref{eq:Faction}
and $b^{\{k\}}=b_1\odot\cdots\odot b_{k-1}\odot b_{k+1}\odot\cdots\odot b_n$. 
Extending this isomorphism of $R$-coalgebras $\OmegaF$-linearly,
we obtain an isomorphism of ${\OmegaF} $-coalgebras
\begin{equation}\label{eq:PBWLiepair}
\overline{\pbw}\colon \OmegaF(S(B)) \to \OmegaF(\cD(B))
.\end{equation}

Transferring the coderivation $d_F^\cU$ 
of $\OmegaF(\cD(B))$ to $\OmegaF(S(B))$ via
the isomorphism \eqref{eq:PBWLiepair},
we obtain a degree~$+1$ coderivation $\deltaa$ of $\OmegaF(S(B))$:
\[ \deltaa :=(\overline{\pbw})^{-1}\circ d_F^\cU\circ \overline{\pbw}:
\OmegaF^\bullet (S(B))\to \OmegaF^{\bullet+1} (S(B)) .\]
Thus
\[ \big( \OmegaF(S(B)), \deltaa, \Delta\big) \]
is a dg coalgebra over $(\Omega_F, d_F)$.

By dualizing $\deltaa$ over the dg algebra $(\Omega_F, d_F)$,
we obtain a degree $+1$ derivation
\begin{equation}\label{eq:D}
D: \OmegaF^\bullet (\widehat{S} (B^\vee)) \to\OmegaF^{\bullet+1}(\widehat{S} (B^\vee)).
\end{equation}

According to \cite[Theorem~5.7]{MR4271478},
$D$ in~\eqref{eq:D} can be expressed as
\[ D=d_{F}^{\nabla^{\Bott}}+\sum_{k=2}^{\infty}\widetilde{\cR}_k ,\]
where
\begin{enumerate}
\item $d_{F}^{\nabla^{\Bott}}$ is the Chevalley--Eilenberg
differential corresponding to the Bott connection of $F$
on $\widehat{S}\big(B^\vee\big)$;
\item for any $k\geq 2$, $\widetilde{\cR}_k:\OmegaF^\bullet(\widehat{S}(B^\vee)) 
\to\OmegaF^{\bullet+1}(\widehat{S}(B^\vee))$
is the $\OmegaF^\bullet$-linear degree $+1$ derivation acting by contraction
induced from a section $\cR_k \in \Omega^1_F (S^k(B^\vee)\otimes B)$;
\item $\cR_2 \in \Omega^1_F(S^2(B^\vee) \otimes B)$ is the Atiyah 1-cocycle
$\At^{\nabla^{\Bott}}_{{\TkM}/F}$
associated with the connection $\nabla^B$ defined by
\[ \cR_2(a,\bar{X}) = \nabla^B_a\nabla^B_X - \nabla^B_X \nabla^B_a - \nabla^B_{[a,X]} ,\]
for all $a\in\Gamma(F)$ and $X\in\Gamma({\TkM})$, where
$\bar{X}\in \sections{B}$ denotes the image of $X$ under the quotient map $\TkM\to\TkM/F$.
\end{enumerate}
A priori, $\cR_2\in\Omega^1_F(B^\vee\otimes\End(B))$,
but the torsion-free assumption guarantees that
it is indeed an element in 
$\Omega^1_F(S^2(B^\vee)\otimes B)$.
Its cohomology class $\alpha_{{\TkM}/F}\in\hypercohomology{1}_{\CE}(F,B^\vee\otimes\End(B))$
is independent of the choice of $\nabla^B$ and is called the Atiyah class of the Lie pair
$(\TkM,F)$ \cite{MR3439229}.
Note that $\OmegaF(\widehat{S}(B^\vee))$ is the algebra of functions
on $F[1]\oplus B$. Thus $(F[1]\oplus B,D)$ is a dg manifold with support $M$,
called a \bfemph{Kapranov dg manifold} associated to the Lie pair $(\TkM,F)$
\cite{MR4271478}.
One can prove that
the various Kapranov dg manifold structures on $F[1]\oplus B$
resulting from all possible choices of splitting and connection are all
isomorphic.

\begin{theorem}[{\cite[Theorem~5.7]{MR4271478}}]\label{thm:HKG}
Let $F \subseteq {\TkM}$ be an integrable distribution.
The choice of a splitting $j\colon B \to {\TkM}$
of the short exact sequence \eqref{SES}
and a torsion-free linear connection $\nabla^B$ of the vector bundle $B$
determines an $L_\infty[1]$ algebra structure on the graded vector
space $\OmegaF^\bullet (B)$
defined by a sequence $(\lambda_k)_{k\in\NN}$ of multibrackets
such that each $\lambda_k$, with $k\geqslant 2$,
is $\OmegaF$-multilinear, and
\begin{itemize}
\item the unary bracket $\lambda_1$ is
the Chevalley--Eilenberg differential $d_F^{\nabla^{\Bott}}$
associated with the Bott connection $\nabla^{\Bott}$ of $F$ on $B$;
\item the binary bracket $\lambda_2$ is the map
\[ \lambda_2: \OmegaF^{j_1}(B)\otimes \OmegaF^{j_2}(B)
\to \OmegaF^{j_1+j_2+1}(B) \]
induced by the Atiyah $1$-cocycle $\cR_2 \in \Omega^1_F(S^2(B^\vee)\otimes B)$;
\item for every $k\geqslant 3$, the $k$-th multibracket $\lambda_k$
is the composition of the wedge product
\[ \OmegaF^{j_1}(B)\otimes\cdots\otimes \OmegaF^{j_k}(B)
\to \OmegaF^{j_1+\cdots+j_k}\big(B^{\otimes k}\big) \]
with the map
\[ \OmegaF^{j_1+\cdots+j_k}\big(B^{\otimes k}\big)\to
\OmegaF^{j_1+\cdots+j_k+1}(B) \]
induced by an element $R_k\in \OmegaF^1\big(S^k({B}\dual)\otimes B\big)
\subset\OmegaF^1 \big({(B\dual)}^{\otimes k}\otimes B \big)\big)$.
\end{itemize}
Moreover, the $L_\infty[1]$ algebra structure
on $\OmegaF^\bullet (B)$ is unique up to isomorphisms
in the sense that those resulting from all possible choices of splitting
and connection are all isomorphic.
\end{theorem}

Any such $L_\infty[1]$ algebra structure 
on $\OmegaF^\bullet (B)$ is called a 
\bfemph{Kapranov $L_\infty[1]$ algebra} of
the integrable distribution $F$.

As a special case, consider a complex manifold $X$.
The subbundle $F = T_X^{0,1} \subset T_\CC X$ is an integrable
distribution, and the normal bundle $B := T_\CC X/T_X^{0,1}$
is naturally identified with $T^{1,0}_X$.
Moreover, the Chevalley--Eilenberg differential associated with
the Bott $F$-connection on $T_X^{1,0}$ becomes the Dolbeault operator
\[ \bar{\partial}\colon \Omega_X^{0,\bullet}(T_X^{1,0}) 
\to \Omega_X^{0,\bullet+1}(T_X^{1,0}) .\]

The following is an immediate consequence of Theorem~\ref{thm:HKG},
which extends Kapranov's construction for Kähler manifolds
\cite[Theorem~2.6]{MR1671737} to all complex manifolds.

\begin{theorem}[{\cite[Theorem~5.24]{MR4271478}}]\label{thm:complex}
For a given complex manifold $X$, any
torsion-free $T\oz_X$-connection $\nabla\oz$ on $T\oz_X$
determines an $L_\infty[1]$ algebra structure on the Dolbeault complex
$\OO^{0,\bullet}(T\oz_X)$ such that
\begin{itemize}
\item the unary bracket $\lambda_1$ is the Dolbeault operator
\[ \overline{\partial}:\Omega^{0,j}(T\oz_X)\to\Omega^{0,j+1}(T\oz_X) ;\]
\item the binary bracket $\lambda_2$ is the map
\[ \lambda_2:\OO^{0,j_1}(T\oz_X)\otimes\OO^{0,j_2}(T\oz_X)
\to \OO^{0, j_1+j_2+1}(T\oz_X) \]
induced by the Dolbeault representative of the Atiyah $1$-cocycle $R_2\in
\Omega\zo\big(S^2{(T\oz_X)}\dual\otimes T\oz_X\big)$;
\item for every $k\geqslant 3$, the $k$-th multibracket $\lambda_k$
is the composition of the wedge product
\[ \OO^{0,j_1}(T\oz_X)\otimes\cdots\otimes\OO^{0,j_k}(T\oz_X)
\to \OO^{0,j_1+\cdots+j_k}\big((T\oz_X)^{\otimes k}\big) \]
with the map
\[ \OO^{0,j_1+\cdots+j_k}\big((T\oz_X)^{\otimes k}\big)
\to\OO^{0,j_1+\cdots+j_k+1}(T\oz_X) \]
induced by an element $R_k$ of
$\Omega\zo\big(S^k\big((T\oz_X)\dual\big)\otimes T\oz_X\big)
\subset\OO\zo\big(({(T\oz_X)}^\vee)^{\otimes k}\otimes T\oz_X\big)$,
which is itself completely determined by 
the Atiyah $1$-cocycle $R_2$, the curvature of $\nabla\oz$, 
and their higher covariant derivatives.
\end{itemize}
Moreover, the $L_\infty[1]$ algebra structure on 
$\Omega^{0,\bullet}(T\oz_X)$ is unique up to isomorphisms.
\end{theorem}

Now we are ready to consider the Kapranov $L_\infty[1]$ algebra of the dg manifold
$(F[1],d_F)$. Let 
\[ \tilde{\Phi} \colon \cD({F[1]}) \to \OmegaF(\cD(B)) \]
be the map defined by 
$\tilde{\Phi}(D)=\overline{\pi_\ast (D)}$,
where $\pi_\ast\colon \cD({F[1]}) \to \OmegaF \otimes_R \cD(M)$ is the pushforward map
\[ \pi_\ast(D)(f) = D(\pi^\ast f),
\qquad\forall D\in\cD({F[1]}), \forall f\in R \]
and $\overline{\pi_\ast (D)} \in \Omega_F(\cD(B))$ denotes the class
of $\pi_\ast (D)$ in $\OmegaF \otimes_R \frac{\cD(M)}{\cD(M)\Gamma(F)} \cong
\OmegaF(\cD(B))$.

\begin{theorem}[\cite{MR3313214,arXiv:2103.08096}]
There exists a contraction of dg $\Omega_F$-modules
\begin{equation}\label{Contration:DFonetoOmegaFDpolyB}
\begin{tikzcd}
(\cD({F[1]}), \UdifferentialM) \arrow[loop left, distance=2em, 
start anchor={[yshift=-1ex]west}, end anchor={[yshift=1ex]west}]{}{\tilde{H}}
\arrow[r, yshift=0.7ex, "\tilde{\Phi}"] & (\Omega_F(\cD(B)),d_F^{\cU})
\arrow[l, yshift=-0.7ex, "\tilde{\Psi}"],
\end{tikzcd}
\end{equation}
where the projection $\tilde{\Phi}$ is a morphism of $\Omega_F$-coalgebras.
\end{theorem}

Choose a torsion-free affine connection $\nabla$ on $F[1]$. We write
\[ \pbwF:\sections{S(\tangent{F[1]})}\to\cD(F[1]) \]
for the corresponding Poincar\'e--Birkhoff--Witt map as in~\eqref{defnpbw}.

By conjugating the PBW maps $\pbwF$ and $\overline{\pbw}$, respectively,
on the left hand side and on the right hand side of~\eqref{Contration:DFonetoOmegaFDpolyB},
we obtain

\begin{corollary}
There exists a contraction of dg $\Omega_F$-modules
\[ \begin{tikzcd}
(\sections{S(T_{F[1]})} , \pbwF^{-1} \circ \UdifferentialM \circ \pbwF)
\arrow[loop left, distance=2em, start anchor={[yshift=-1ex]west},
end anchor={[yshift=1ex]west}]{}{{H}}
\arrow[r, yshift = 0.7ex, "{\Phi}"]
& (\Omega_F(S(B)),{\overline{\pbw}}^{-1}\circ d_F^{\cU}\circ\overline{\pbw})
\arrow[l, yshift = -0.7ex, "\Psi"],
\end{tikzcd} \]
where the projection ${\Phi}:={\overline{\pbw}}^{-1}\circ\tilde{\Phi}\circ\pbwF$
is a morphism of $\Omega_F$-coalgebras.
\end{corollary}

The projection $\Phi$ determines a sequence of maps
$\{\phi_k\}_{k\geq 1}$ making the diagrams
\begin{equation}\label{diagram:phik}
\begin{tikzcd}
S^k_\KK\big(\XX(F[1])\big) \arrow[d] \arrow[r, "\phi_{k}"] & \Omega_F (B)	\\
\sections{S(T_{F[1]})} \arrow[r, "\Phi"] & \Omega_F\big(S(B)\big) \arrow[u]
\end{tikzcd}
\end{equation}
commutative.
Note that $\phi_1:\XX( F[1])\to\Omega_F (B)$ is the composition
\[ \XX(F[1])\xto{\pi_*}\Omega_F(\TkM)\xto{\prB}\Omega_F(B) .\]

\begin{theorem}
Let $F \subseteq {\TkM}$ be an integrable distribution.
Then the sequence of $\Omega_F$-multilinear
maps $\{\phi_k\}_{k\geq 1}$
defined by the commutative diagrams \eqref{diagram:phik}
constitutes a quasi-isomorphism from the 
Kapranov $L_\infty[1]$ algebra $\XX(F[1])$ 
arising from the dg manifold $(F[1],d_F)$
to the Kapranov $L_\infty[1]$ algebra 
$\Omega_F^\bullet(B)$ arising (as in
Theorem~\ref{thm:HKG}) from the integrable distribution $F$.
\end{theorem}

As an immediate consequence, we have
\begin{corollary}\label{cor:NYC}
For any complex manifold $X$, consider its corresponding
dg manifold $(T^{0,1}_X[1], \bar{\partial})$ as in Example~\ref{example-two}.
The Kapranov $L_\infty[1]$ algebra
$\XX(T_X^{0, 1}[1])$ is quasi-isomorphic to
the $L_\infty[1]$ algebra
$\OO^{0,\bullet}(T\oz_X)$---see~Theorem~\ref{thm:complex}.
The quasi-isomorphism $\{\phi_k\}_{k\geq 1}$, 
in which each map $\phi_k$ is $\OO^{0,\bullet}_X$-multilinear, is given by
\eqref{diagram:phik} (with $F=T_X^{0,1}$ and $B=T\oz_X$),
and in particular, the linear part
$\phi_1: \XX(T_X^{0, 1}[1])\to \OO^{0, \bullet}(T\oz_X)$
is given by the composition
\[ \XX(T_X^{0, 1}[1])\xto{\pi_*}\Omega^{0,\bullet}(T_X^\CC)\xto{\pr}\OO^{0,\bullet}(T\oz_X) .\]
\end{corollary}

\appendix

\section{Fedosov construction on graded manifolds}
\label{sec:A}
This section is to give a brief description of
Fedosov construction of graded manifolds. We refer readers to
\cite{MR2102846,MR1327535,MR3910470} for more details.

Throughout this section, $\cM$ is a finite dimensional graded manifold
and $\nabla$ is a torsion-free affine connection on $\cM$.
There is an induced linear connection on $\widehat{S}(\tangent{\cM}^{\vee})$,
which is denoted by the same symbol
$\nabla$ by abuse of notation.

Consider the map $\nabla^{\lightning}:\XX(\cM)\times\sections{S(\tangent{\cM})}
\to\sections{S(\tangent{\cM})}$
\[\nabla^{\lightning}_{Y}\boldX = (\pbw^{\nabla})^{-1}
\big( Y\cdot \pbw^{\nabla}(\boldX)\big) \]
for any $Y\in\XX(\cM)$ and $\boldX\in\sections{S(\tangent{\cM})}$.

\begin{lemma}
The above map $\nabla^{\lightning}$ defines
a flat connection on $S(\tangent{\cM})$.
\end{lemma}

Abusing notation, we write the same symbol $\brevee{\nabla}^{\lightning}$
to denote the induced flat connection on $\widehat{S}(\tangent{\cM}^{\vee})$.
Then the associated covariant derivative $d^{\brevee{\nabla}^{\lightning}}$ satisfies
$(d^{\brevee{\nabla}^{\lightning}})^{2}=0$.

In the following, we use the identification
\[\Omega^{p}(\widehat{S}(\tangent{\cM}^{\vee})) \cong
\sections{\Lambda^{p}(\tangent{\cM}^{\vee})\otimes \widehat{S}(\tangent{\cM}^{\vee})} \cong
\sections{\Hom(\Lambda^{p}(\tangent{\cM})\otimes S(\tangent{\cM}),\KK)}\]
and the total degree of $\omega \in \Omega^{p}(\widehat{S}(\tangent{\cM}^{\vee}))$
is $ p+\degree{\omega}$,
where $p$ is the cohomological degree and $\degree{\omega}$ is the internal degree.

Define two operators
\[ \delta:\Omega^p(\widehat{S}(T\dual_\cM))
\to\Omega^{p+1}(\widehat{S}(T\dual_\cM)) \]
and \[ \frakh:\Omega^p(\widehat{S}(T\dual_\cM))
\to\Omega^{p-1}(\widehat{S}(T\dual_\cM)) \] by
\[ \left(\delta\omega\right)(X_1\wedge\cdots\wedge X_{p+1};
Y_1\odot\cdots\odot Y_{q-1})=\sum_{i=1}^{p+1} (-1)^{i+1} \sign \cdot
\omega(X_1\wedge\cdots\wedge\widehat{X}_i\wedge\cdots\wedge X_{p+1};
X_i\odot Y_1\odot\cdots\odot Y_{q-1}) \]
and
\[ \left(\frakh\omega\right)(X_1\wedge\cdots\wedge X_{p-1};
Y_1\odot\cdots\odot Y_{q+1})=\frac{1}{p+q}\sum_{j=1}^{q+1}
\sign \cdot \omega(Y_j\wedge X_1\wedge\cdots\wedge X_{p-1};
Y_1\odot\cdots\odot\widehat{Y}_j\odot\cdots\odot Y_{q+1}) ,\]
for all $\omega\in\Omega^{p}(\widehat{S}(\tangent{\cM}^{\vee}))$
and all homogeneous $X_1,\cdots,X_{p+1},Y_1,\cdots,Y_{q+1}\in\XX(\cM)$.
The symbol $\sign$ denotes the Koszul signs:
either $\sign(X_1,\cdots,X_{p+1},Y_1,\cdots,Y_{q-1})$
or $\sign(X_1,\cdots,X_{p-1},Y_1,\cdots,Y_{q+1})$, as appropriate.

Both $\delta$ and $\frakh$ are $\smooth{\cM}$-linear,
and $\delta$ is the Koszul operator.
Observe that $\delta$ has total degree $+1$ and $\frakh$ has total degree $-1$.
However neither $\delta$ nor $\omega$ change the internal degree:
$\degree{\delta\omega}=\degree {\omega}$ and $\degree{\frakh\omega}=\degree{\omega}$
for $\omega\in\Omega^{p}(\widehat{S}(\tangent{\cM}^{\vee}))$.

\begin{remark}
In \cite{MR2102846,MR1327535,MR3910470}, the operator $\frakh$ is written as $\delta\inv$.
We avoid this notation because $\frakh$ is not an inverse map of $\delta$,
and it is rather a homotopy operator.
\end{remark}

\begin{lemma}\label{lem:delta}
The operator $\delta$ satisfies $\delta^{2}=0$. That is,
\[0\to \Omega^{0}(\widehat{S}(\tangent{\cM}^{\vee})) 
\xrightarrow{\delta} \Omega^{1}(\widehat{S}(\tangent{\cM}^{\vee}))
\xrightarrow{\delta}\Omega^{2}(\widehat{S}(\tangent{\cM}^{\vee})) \xrightarrow{\delta} \cdots\]
forms a cochain complex. Moreover, it satisfies
\[\delta \circ \frakh +\frakh\circ \delta = \id - \pi_{0}\]
where $\pi_{0}:\Omega^{\bullet}(\widehat{S}(\tangent{\cM}^{\vee}))\to C^{\infty}(\cM)$
is the natural projection.
\end{lemma}

We have the following theorem
\begin{theorem}[{\cite[Theorem~5.6]{MR3910470}}]\label{LStheorem}
Let $\cM$ be a finite dimensional graded manifold and $\nabla$ be
a torsion-free affine connection on $\cM$.
Then the covariant derivative $d^{\brevee{\nabla}^{\lightning}}$ decomposes as
\[ d^{\brevee{\nabla}^{\lightning}}=d^{\brevee{\nabla}}-\delta+\widetilde{\Amap} ,\]
where the operator $\widetilde{\Amap}:
\Omega^{\bullet}\big(\widehat{S}(\tangent{\cM}^{\vee})\big)
\to\Omega^{\bullet+1}\big(\widehat{S}(\tangent{\cM}^{\vee})\big)$,
is a (total) degree $+1$ derivation determined
by $\Amap\in\Omega^{1}\big(\cM,\widehat{S}^{\geq 2}(\tangent{\cM}^{\vee})
\otimes\tangent{\cM}\big)$, satisfying \[ \frakh\circ\Amap=0 .\]
\end{theorem}

\begin{remark}
The operator $\widetilde{A^{\nabla}}$ increases the cohomological degree by $+1$
while it preserves the internal degree. That is, although the total degree
of $\widetilde{A^{\nabla}}$ is $+1$, we have the internal degree
$\degree{\widetilde{A^{\nabla}}}=0$.
\end{remark}
Write
\[\Amap = \sum_{n\geq 2}\Amap_{n}, \quad \quad \Amap_{n}
\in\Omega^{1}(\cM,S^{n}(\tangent{\cM}^{\vee})\otimes\tangent{\cM}).\]

Let $R^{\brevee{\nabla}}\in\Omega^{2}\big(\cM;\End(\tangent{\cM})\big)$
denote the curvature of $\brevee{\nabla}$.

\begin{proposition}
We have the following recursive formula for $\Amap_{n}$:
\begin{gather*}
\Amap_{2}=\frakh \circ R^{\nabla}\, ,\\
\Amap_{n+1}=\frakh \circ \left( d^{\nabla}\Amap_{n}
+\sum_{p+q=n}\half [\Amap_{p}, \Amap_{q}] \right), \quad \forall n\geq 2.
\end{gather*}
\end{proposition}

\begin{proof}
By Theorem~\ref{LStheorem}, the covariant derivative
$d^{\brevee{\nabla}^{\lightning}} = d^{\brevee{\nabla}}-\delta+\Amap$
and satisfies $(d^{\brevee{\nabla}^{\lightning}})^{2}=0$.

By Lemma~\ref{lem:delta}, we know $\delta^{2}=0$
and $\delta\circ \frakh+\frakh \circ \delta = \id-\pi_{0}$. Also,
$(d^{\brevee{\nabla}})^{2}=R^{\brevee{\nabla}}$.
Since ${\nabla}$ is torsion-free, we have
\[[\delta, d^{\brevee{\nabla}}]=\delta\circ d^{\brevee{\nabla}}
+d^{\brevee{\nabla}}\circ \delta = 0.\]
As a result, $(d^{\brevee{\nabla^{\lightning}}})^{2}=0$ implies that
\[\delta\circ \Amap + \Amap\circ \delta = R^{\brevee{\nabla}} +
d^{\brevee{\nabla}}\Amap + \half [\Amap, \Amap] \]
By applying the operator $\frakh$, we get
\[\Amap = \frakh \circ \delta \circ \Amap = \frakh \circ \left(
R^{\brevee{\nabla}}+d^{\brevee{\nabla}}\Amap+ \half [\Amap, \Amap] \right)\]
because $ \frakh \circ \Amap=0$ and $\pi_{0}\circ \Amap=0$.

Since $\frakh\big( \Omega^2(\widehat{S}^q( T\dual_\cM))
\subset \Omega^{1}(\widehat{S}^{q+1}( T\dual_\cM))$,
applying the canonical projections
\[ \Omega^{1}(\cM,\hat{S}(\tangent{\cM}^{\vee})\otimes\tangent{\cM})
\to\Omega^{1}(\cM, S^{n}(\tangent{\cM}^{\vee})\otimes\tangent{\cM}) \]
(for each $n\geq 2$) to the equality
\[\Amap = \frakh \circ \left(
R^{\brevee{\nabla}}+d^{\brevee{\nabla}}\Amap+ \half [\Amap, \Amap] \right) 
\in \Omega^{1}(\cM,\hat{S}(\tangent{\cM}^{\vee})\otimes\tangent{\cM}) \]
yields the relations 
\begin{gather}
\Amap_{2}=\frakh \circ R^{\brevee{\nabla}} , \nonumber \\
\Amap_{n+1}=\frakh \circ \left( d^{\brevee{\nabla}}\Amap_{n}
+\sum_{p+q=n}\half [\Amap_{p}, \Amap_{q}] \right), \quad \forall n\geq 2. \label{eq:A}
\end{gather}
This completes the proof.
\end{proof}

\begin{corollary}\label{cor:An}
Under the same hypothesis as in Theorem~\ref{LStheorem}, the element
$\Amap_{n}\in\Omega^{1}(\cM,S^{n}(\tangent{\cM}^{\vee})\otimes\tangent{\cM})$, with $n\geq 2$,
is completely determined by the curvature $\curvature$ and its higher covariant
derivatives. In fact, $\Amap_{n} $ satisfies the recursive formula
\eqref{eq:A} involving $\Amap_k$, with $k\leq n-1$.
\end{corollary}

\section*{Acknowledgments}

We would like to thank Ruggero Bandiera,  
Camille Laurent-Gengoux, Hsuan-Yi Liao, Rajan Mehta
and Luca Vitagliano for fruitful discussions and useful comments.
Seokbong Seol is grateful to the Korea Institute for Advanced Study
for its hospitality and generous support.

\printbibliography
\end{document}